\documentclass[twoside,12pt,reqno]{amsart}
\usepackage{a4wide}
\usepackage[utf8]{inputenc}
\usepackage[T1]{fontenc}
\usepackage[english]{babel}
\usepackage{tgpagella}
\usepackage{hyperref}
\usepackage{amscd} 
\usepackage{csquotes}
\usepackage{fancyhdr}
\usepackage{multirow}
\usepackage{cleveref}
\usepackage{comment}

\usepackage{tikz}
\usepackage{pgf,tikz}
\usepackage{tikz-cd,rotating}

\usepackage{amsthm, amsfonts, amssymb, amsmath, latexsym, enumerate,array,color}
\usepackage{amscd} 
\usepackage[all]{xy}
\usepackage{pstricks,graphicx}
\usepackage{stmaryrd}
\usepackage{hyperref,mdwlist,mathrsfs}
\usepackage{tikz}
\usepackage{pgf,tikz}
\usetikzlibrary{arrows}

\usepackage{color}

\usepackage{amssymb}
\usepackage{amsmath}
\usepackage{amsthm}
\usepackage{paralist}

\newcommand{\cB}{\ensuremath{\mathcal{B}}}
\newcommand{\HF}{\ensuremath{\mathrm{HF}}}

\newtheorem{theorem}{Theorem}[section]
\newtheorem{lemma}[theorem]{Lemma}

\newtheorem{proposition}[theorem]{Proposition}
\newtheorem{corollary}[theorem]{Corollary}
 
\newtheorem*{proposition*}{Proposition}
\theoremstyle{definition}
\newtheorem{definition}[theorem]{Definition}

\newtheorem{remark}[theorem]{Remark}
\newtheorem{notation}[theorem]{Notation}

\DeclareMathOperator{\Ann}{Ann}

\DeclareMathOperator{\rank}{rank}

\DeclareMathOperator{\Hess}{Hess}

\DeclareMathOperator{\reg}{reg}

\numberwithin{equation}{section}

\title[The WLP for artinian Gorenstein algebras of small Sperner number]{The weak Lefschetz property for artinian Gorenstein algebras of small Sperner number}

\author[]{Mats Boij}
\address{Mats Boij: Department of Mathematics, KTH Royal Institute of Technology, 100 44 Stockholm, Sweden}
\email{boij@kth.se}

\author[]{Juan C. Migliore}
\address{Juan C. Migliore: Department of Mathematics, University of Notre  Dame, Notre Dame, IN 46556, USA}
\email{migliore.1@nd.edu}

\author[]{Rosa M.\ Mir\'o-Roig}
\address{Rosa Maria Mir\'o-Roig: Facultat de
Matem\`atiques i Inform\`atica, Universitat de Barcelona, Gran Via des les Corts Catalanes 585, 08007 Barcelona, Spain}
\email{miro@ub.edu,  ORCID 0000-0003-1375-6547}

\author[]{Uwe Nagel} 
\address{Uwe Nagel: Department of Mathematics, University of Kentucky, 715 Patterson Office Tower, Lexington, KY 40506-0027, USA}
\email{uwe.nagel@uky.edu}

\thanks{\hspace{-15pt} Boij was a Visiting Scholar at UC Berkeley while this work was carried out,}

\thanks{\hspace{-15pt} Migliore was partially supported by Simons Foundation grant \#839618,}
  
\thanks{\hspace{-15pt}   Mir\'o-Roig was partially supported by the grant PID2020-113674GB-I00,  
}

\thanks{\hspace{-15pt} Nagel was partially supported by Simons Foundation grant \#636513}

\theoremstyle{definition}

\begin{document}

\begin{abstract} 
For artinian Gorenstein algebras in codimension four and higher, it is well known that the Weak Lefschetz Property (WLP) does not need to hold. For Gorenstein algebras in codimension three, it is still open whether all artinian Gorenstein algebras satisfy the WLP when the socle degree and the Sperner number are both higher than six. We here show that all artinian Gorenstein algebras with socle degree $d$ and Sperner number at most $d+1$ satisfy the WLP, independent of the codimension. This is a sharp bound in general since there are examples of artinian Gorenstein algebras with socle degree $d$ and Sperner number $d+2$ that do not satisfy the WLP for all $d\ge 3$. 
\end{abstract}
\maketitle

\section{Introduction} 
\vskip 4mm
Artinian standard graded algebras have long been ubiquitous in many fields of mathematics, including algebraic geometry, commutative algebra, algebraic topology, combinatorics, etc. Over the last several decades, one general aspect of such algebras that has been extensively studied is whether or not they have the Weak Lefschetz Property (WLP -- see Definition \ref{WLP def}). If they do, it has strong implications for the Hilbert function; for example, not only must these functions be unimodal but in fact they must satisfy a strong growth condition \cite{HMNW}. In the other direction, it is also known precisely which Hilbert functions force the WLP to hold, with no additional information needed about the algebra \cite{MZ3}. 

While the role of the characteristic plays a very interesting role in such questions, in this paper we assume throughout that our work is over a field of characteristic zero.

One of the most important kinds of algebras that have been studied from the point of view of WLP is the class of artinian Gorenstein algebras. (From now on we will always assume that our algebras are graded.) Since 1980, many papers appeared showing that the Hilbert function of an artinian Gorenstein algebra need not be unimodal, including the initial result by R. Stanley \cite{stanley} who produced the first example, namely the Hilbert function $(1,13,12,13,1)$,  and the first author of this paper, who showed that the Hilbert function of an artinian Gorenstein algebra could be as ``non-unimodal" as one wants \cite{B2}. Such algebras clearly fail the WLP. 

Taking a cue from the situation with artinian algebras in general, an intriguing problem is to look for conditions that force an artinian Gorenstein algebra to have the WLP. An important special case is an old conjecture \cite{RRR} that all complete intersections have the WLP. This is currently known only in two or three variables \cite{HMNW}, although partial results are known. In this paper we are interested in properties of the Hilbert function that force an artinian Gorenstein algebra (not necessarily a complete intersection) to have the WLP. In this case the Hilbert function has the form
\[
(1,h_1,h_2,\dots, h_{d-1}, 1)
\]
and it is further known that this sequence is symmetric ($h_i = h_{d-i}$). The maximum entry of this vector is called the \emph{Sperner number} of the algebra.

It is known that the condition $h_1=2$ forces the WLP among artinian Gorenstein algebras since all algebras in two variables have the SLP. When $h_1=3$, it is conjectured that all artinian Gorenstein algebras have the WLP, so the condition $h_1 = 3$ is conjectured to be a condition that forces WLP. Partial results are known, for instance in \cite{BMMNZ} it was shown by the authors together with F.~Zanello that all artinian Gorenstein algebras with socle degree $d\le 6$ and codimension $h_1=3$ satisfy the SLP. More recently, this was extended by N.~Abdallah, N.~Altafi, A.~Iarrobino, A.~Seceleanu and J.~Yam\'eogo~\cite{AAISY} to all artinian Gorenstein algebras of codimension $h_1=3$ and Sperner number at most $6$. 

It is well known that artinian Gorenstein algebras failing WLP exist for all $h_1 \geq 4$. In the other extreme, trivially (using duality) it is true that any compressed Gorenstein algebra (i.e. one of maximal Hilbert function) of even socle degree has the WLP;  e.g. $(1,4,10,20,10,4,1)$ has it, since in the first half it inherits injectivity from the polynomial ring, and the second half follows by duality. The analogous statement for odd socle degree is not true for any $h_1 \geq 4$. So it is of interest to find more subtle properties of the Hilbert function that force WLP to hold. 
For instance, it follows easily from the theorem of Gordan and Noether (see Theorem \ref{hesscriterium}  and Corollary \ref{wlp slp}) that in characteristic zero a Hilbert function $(1, 4, k, k, 4, 1)$ forces not only  the WLP but also the SLP. On the other hand, as an illustration of the role of the characteristic, Kustin \cite{kustin} has shown that the Hilbert function $(1,4,4,1)$ has the WLP in all characteristics except two, and in the latter case he has given the unique ideal (up to change of variables) that fails. In this paper, while we restrict to characteristic zero, the identification of ideals up to change of variables that have certain properties plays an important role.

The goal of this paper is to prove a very broad result, giving a class of Hilbert functions in any number of variables $\geq 3$ that force artinian Gorenstein algebras to have the WLP. To be precise, we prove the following theorem. 
\begin{theorem}[Theorem~\ref{mainthm}]
  Any artinian Gorenstein algebra $A$ of socle degree $d$ and Sperner number $\le d+1$ satisfies the WLP.  
\end{theorem}
This result is optimal in the sense that there are artinian Gorenstein algebras failing to have the WLP with socle degree $d$ and Sperner number $d+2$ for all $d\ge3$. These examples can be obtained by Stanley's construction by trivial extensions also known as \emph{Perazzo algebras} (cf. Remark~\ref{optimal}).
In order to prove our main theorem, we introduce some surprisingly strong new methods, adapt some known methods, and give an induction on the socle degree of the algebra. We now give a more precise description of the main results and the structure of this paper. 

In Section~\ref{S2} we recall necessary definitions and theorems from the literature needed for the following sections. The proof of the main theorem is build on an induction on the socle degree and by establishing the WLP for three families of artinian Gorenstein algebras with the following three $h$-vectors:
\begin{itemize}
    \item $(1,3,5,7,9,\dots,2m-1,2m+1,2m+1,2m-1,\dots,9,7,5,3,1)$, 
    \item $(1,3,6,8,10,\dots,2m,2m+2,2m+2,2m,\dots,10,8,6,3,1)$, and 
    \item $(1,4,6,8,10,\dots,2m,2m+2,2m+2,2m,\dots,10,8,6,4,1)$.
\end{itemize}
In Section~\ref{S3} we prove these results. The proofs of the WLP for the first two families use results by Migliore, Nagel and Zanello~\cite{MNZ} and by Migliore and Zanello~\cite{MZ} about the WLP for certain artinian Gorenstein algebras in codimension three and four when the socle degree is at least $7$ and the results from \cite{BMMNZ} when the socle degree is at most $6$. The proof of the WLP for the third family is the main effort of this paper and it builds on considering two cases.
The first case is to reduce to the second family by constructing a subalgebra generated by three general linear forms. The second case deals with the situation where this reduction does not work. In this case the initial monomials with respect to the lexicographic order of the quadrics of our Gorenstein ideal have to be $x_1^2,x_1x_2,x_1x_3,x_2^2$ after a generic change of coordinates. In Theorem~\ref{thm:quadrics} we provide a complete classification of such ideals generated by four quadrics, not assuming that they are Gorenstein.
Our results show that there is a finite number of orbits under the general linear group, and we identify explicitly an orbit generator for each orbit.
The analysis of the orbits is based on a geometric approach. Once this is established, we prove the WLP for the artinian Gorenstein algebras in the third family for one representative from each orbit. 
In Section~\ref{S4} we prove the main theorem and as a corollary we give a complete classification of the $h$-vectors of Sperner number at most $6$ that force the WLP, see \Cref{low sperner}.

\vskip 4mm
\noindent {\bf Acknowledgements.} 
This paper began as a Research in Teams project at the Banff International Research Station for Mathematical Innovation and Discovery in March, 2024. 
The authors are very  grateful to BIRS for providing the opportunity and wonderful setting for this collaboration. 

\section{Background material}\label{S2}
Throughout this paper $K$ will be an algebraically closed field of characteristic zero.
Given a standard graded artinian $K$-algebra $A=R/I$ where $R=K[x_1,\dots,x_n]$ and $I$ is a homogeneous ideal of $R$, we denote by $\HF_A\colon \mathbb{Z} \longrightarrow \mathbb{Z}$ with $\HF_A(j)=\dim _K[A]_j=\dim _K[R/I]_j$ its Hilbert function. Since $A$ is artinian, its Hilbert function is captured in its \emph{$h$-vector} $h=(h_0,\dots ,h_d)$ where $h_i=\HF_A(i)>0$ and $d$ is the largest index with this property. The integer $d$ is called the \emph{socle degree of} $A$ or the regularity of $A$ and denoted $\reg(A)$. 


\subsection{Lefschetz properties}

In this section we give the definitions of the weak and strong Lefschetz properties. 

\begin{definition} \label{WLP def}
Let $A=R/I$ be a graded artinian $K$-algebra. We say that $A$ has the {\em weak Lefschetz property} (WLP, for short)
if there is a linear form $\ell \in [A]_1$ such that, for all
integers $i\ge0$, the multiplication map
\[
\times \ell: [A]_{i}  \longrightarrow  [A]_{i+1}
\]
has maximal rank.
 In this case, the linear form $\ell$ is called a {\it (weak) Lefschetz element} of $A$. If for the general form $\ell \in [A]_1$ and for an integer $j$ the map $\times \ell:[A]_{j-1}  \longrightarrow  [A]_{j}$ does not have maximal rank, we will say that the ideal $I$ fails the WLP in degree $j$.

$A$ has the {\em strong Lefschetz property} (SLP, for short) if there is a linear form $\ell \in [A]_1$ such that, for all integers $i\ge0$ and $k\ge 1$, the multiplication map
\[
\times \ell^k: [A]_{i}  \longrightarrow  [A]_{i+k}
\]
has maximal rank.  Such an element $\ell$ is called a {\it  strong Lefschetz element} for $A$.
\end{definition}


 \subsection{Gorenstein algebras and Macaulay duality}

\begin{definition}
    A standard graded artinian algebra $A = R/I$ is {\it Gorenstein} if its $h$-vector is symmetric and its socle is one-dimensional, occuring in just the last degree. We sometimes refer to the number of variables of $R$ as the {\it codimension} of $A$.
\end{definition}

We quickly recall the construction of the standard graded artinian Gorenstein algebra $A_F$ with {\em Macaulay dual generator} a given form $F\in S=K[X_0,\dots,X_N]$; we denote by $R=K[x_0,\ldots,x_N]$ the ring of differential operators acting on the polynomial ring $S$, i.\,e.\ $x_i=\frac{\partial}{\partial X_i}$. Therefore $R$ acts on $S$ by differentiation. Given polynomials $p\in R$ and $G\in S$ we will denote by ${p\circ G}$ the differential operator $p$ applied to $G$. For a homogeneous polynnomial $F \in S$ we define
\[
\Ann_RF:=\{p\in R \mid p\circ F=0\}\subset R,
\]
and $A_F=R/\Ann_R F$: it is a standard graded artinian Gorenstein $K$-algebra and $F$ is called its {\em Macaulay dual generator}. It is worthwhile to point out that every standard graded artinian Gorenstein $K$-algebra is of the form $A_F$ for some form $F$, in view of the \lq\lq Macaulay double annihilator Theorem'' (see for instance \cite[Lemma 2.12]{IK}). For an ideal $I\subseteq R$, the \emph{inverse system} $I^\perp$, is given by 
\[
I^\perp := \{F\in S\colon g\circ F = 0, \forall g\in I\}.
\]
\begin{lemma} \label{basic seq}
\label{exactsequence}
Let $A=A_F$ be an artinian Gorenstein $K$-algebra and set $I=\Ann_R(F)$. Then for every linear form $\ell\in A_1$ the sequence
\begin{equation}
\label{seq}
0\longrightarrow \frac{R}{(I:\ell)}(-1)\longrightarrow \frac{R}{I}\longrightarrow \frac{R}{(I,\ell)}\longrightarrow 0
\end{equation}
is exact. Moreover $\frac{R}{(I:\ell)}$ is an artinian Gorenstein algebra with $\ell\circ F$ as dual generator.
\end{lemma}
\begin{proof}
We get the result cutting the exact sequence
\[
0\longrightarrow \frac{(I:\ell)}{I}(-1)
\longrightarrow \frac{R}{I}(-1)\xrightarrow{\,\,\,\times\ell\,\,\,}\frac{R}{I}\longrightarrow \frac{R}{(I,\ell)}\longrightarrow 0
\]
into two short exact sequences. The second fact is a straightforward computation.
\end{proof}

\begin{notation} \label{ABC}
Let $A = R/I$ be an artinian Gorenstein algebra and let $\ell$ be a general linear form.
The sequence given in Lemma \ref{basic seq} will be used often in this paper, so we introduce the following notation.
\[
B = \frac{R}{(I:\ell)} = \frac{A}{(0:\ell)}\ \ \ \hbox{and} \ \ \ C = \frac{R}{(I,\ell)}=\frac{A}{(\ell)}.
\]
It is useful to note that $B$ is isomorphic tothe image of the map $A(-1) \stackrel{\times \ell}{\longrightarrow} A$.
\end{notation}
 

\subsection{Useful theorems}

\begin{remark}
Let $n$ and $i$ be positive integers. The {\em i-binomial expansion of n} is 
\[n_{(i)}=\binom{n_i}{ i}+\binom{n_{i-1}}{ i-1}+...+\binom{n_j}{ j},\] 
where $n_i>n_{i-1}>...>n_j\geq j\geq 1$. Such an expansion always exists and it is unique (see, e.g.,  \cite[Lemma 4.2.6]{BH}).

Following \cite{BH}, we define, for any integers $a$ and $b$,
\[
(n_{(i)})_{a}^{b}=\binom{n_i+b}{ i+a}+\binom{n_{i-1}+b}{ i-1+a}+...+\binom{n_j+b}{ j+a},
\]
where we set $\binom{m}{ q}=0$ whenever $m<q$ or $q<0$.
\end{remark} 

\begin{theorem} \label{thm:useful}
Let $A = R/I$ be a standard graded $K$-algebra and let $\ell \in A$ be a general linear form. Denote by $h_d$  the degree $d$ entry of the Hilbert function of $A$ and by $h_d^{'}$ the degree $d$ entry of the Hilbert function of $A/\ell A \cong R/(I,\ell)$. Then:
\begin{itemize}
\item[(i)] { (Macaulay)} 
\[
h_{d+1}\leq ((h_d)_{(d)})_1^1.
\]
\item[(ii)] { (Green)} 
\[
h_d^{'}\leq ((h_d)_{(d)})_{0}^{-1}.
\]
\item[(iii)] (Gotzmann) 
If $h_{d+1} = ((h_d)_{(d)})^{+1}_{+1}$ and $I$ is generated in degrees $\leq d+1$, then 
\[
h_{d+s} = ((h_d)_{(d)})^s_s \hbox{\hspace{.3cm} for all $s \geq 1$}.
\]
\end{itemize}
\end{theorem}

\begin{proof}
   For (i) see \cite[Theorem 4.2.10]{BH}. 
    For (ii) see \cite[Theorem 1]{G}. For (iii) see     \cite[Theorem 4.3.3]{BH}, or \cite{Go}.
\end{proof} 

The following observation will be  used many times in this paper. 

\begin{lemma} \label{lem:hilb is two}
If $A$ is a standard graded $K$-algebra with $\dim_K [A]_j \le 2j$ for some integer $j$ then $\dim_K [A/\ell A]_j \le 1$ for any general $\ell \in [A]_1$. 
\end{lemma}

\begin{proof}
Notice that 
\[
2 j = \binom{j+1}{j} + \binom{j-1}{j-1} + \cdots + \binom{1}{1}. 
\]
Hence, the claim follows from Green's \Cref{thm:useful}(ii).
\end{proof}

\begin{definition}
If $B = R/I$ is a graded algebra and $A = B/J$ is an artinian Gorenstein quotient of $B$ of socle degree $d$ we can write $A =  R/\Ann_R(F)$ where $F\in [I^\perp]_d$. The $i$th \emph{catalectiacant matrix} is given by 
    \[
    (\operatorname{Cat}^i(F))_{u,v}= (\alpha^{(i)}_u\beta^{(d-i)}_v)\circ F
    \]
    where $\{\alpha^{(i)}_m \}_m$ is a $K$-basis for $[B]_i$ and $\{\beta^{(d-i)}_m \}_m$ is a $K$-basis for $[B]_{d-i}$.  
\end{definition}

It is well-known (cf. \cite[§1]{IK}) that the rank of the catalecticant matrices equals the Hilbert function of $A$: 
\[
\dim_K [A]_i = \rank \operatorname{Cat}^i(F), \quad i = 0,1,\dots,d.
\]

\begin{definition}
  For a homogeneous $F\in S$, let $A = R/\Ann_RF$ be its associated artinian Gorenstein algebra. Let $\cB_i = \{\alpha^{(i)}_m \}_m$ be a $K$-basis of $[A]_i$. The entries of the $i$-th \emph{Hessian matrix} of $F$ with respect to $\cB_i$ are given by
\[
(\Hess^i(F ))_{u,v} = (\alpha^{(i)}_u \alpha^{(i)}
_v )\circ F.
\]
\end{definition}

Up to a non-zero constant multiple the determinant $\det (\Hess^i(F ))$ is independent of the basis $\cB_i$. 
T. Maeno and J. Watanabe provided a criterion for artinian Gorenstein  algebras to have the SLP and to identify the SL elements using the Hessians. Indeed, we have:

\begin{theorem}\label{Hess}
Let $F\in S$ be a homogeneous polynomial of degree $d$, let  $ A= R/\Ann_R F$ be the associated artinian Gorenstein algebra and let $\ell=a_1x_1+\cdots +a_nx_n\in [A]_1$ be a linear form.
Then, up to a multiplicative constant, the $i$-th Hessian matrix $\Hess^i(F)(a_1,\dots, a_n)$ is the matrix of the dual map of the multiplication map $\times \ell^{d-2i}\colon [A]_{i}  \to  [A]_{d-i}$.
Therefore, $\ell $ is a SL element of $A$ if and only if
\[
\det( \Hess ^i _F(a_1,\dots, a_n))\ne 0,\quad \text{for each} \  i=0,\dots,\left\lfloor \tfrac{d}{2}\right\rfloor.
\]        

\end{theorem}
\begin{proof}
    See, for instance, \cite[Theorem 3.1]{MW}.
\end{proof}

The following theorem is due to P. Gordan and M. Noether in \cite{GN}, and it was reproved in
\cite{WatdeB}.

\begin{theorem}\label{hesscriterium}  
Assume $n\le 4$. For any form $F\in S$ the Hessian determinant $\det (\Hess^1(F ))$ is identically zero if and only if F is a cone: that is, if F is annihilated by a linear form  in the Macaulay duality.
\end{theorem}

In particular, this implies by Theorem \ref{Hess} that, setting $d = \deg(F )$, $I =\Ann_RF$ with $I \subset m^2$, and $A = R/I$ for a polynomial ring $R$ of codimension $n\le 4$, the multiplication map $\times \ell ^{d-2} : [A]_1 \rightarrow [A]_{d-1}$ is an isomorphism for a general linear form $\ell \in [A]_1$. We get the following immediate consequence of Theorem \ref{hesscriterium}.

\begin{corollary} \label{wlp slp}
    If $A = R/I$ is an artinian Gorenstein algebra with $h$-vector $(1,4,k,k,4,1)$ then $A$ has the WLP if and only if it has the SLP. 
\end{corollary}

We recall the following three results that will be used in the next section. 

\begin{proposition} [\cite{MNZ}] \label{MNZ result}
Let $R = K[x_1,x_2,x_3,x_4]$ and let $J = (f,g_1,g_2) \subset R$ be a homogeneous ideal with three minimal generators, where $\deg f = a \geq 2$ and $\deg g_1 = \deg g_2 = b \geq a$. Let $\ell_1, \ell_2 \in R$ be generic linear forms. Then $\dim[R/(J,\ell_1,\ell_2)]_b = a-1$ if and only if $f,g_1,g_2$ have a GCD of degree $a-1$. Otherwise $\dim [R/(J,\ell_1,\ell_2)]_b = a-2$.
\end{proposition}

\begin{lemma} [\cite{MZ}] \label{MZ Ill}
Let $R = K[x_1,x_2,x_3]$ and let  $R/I$ be an artinian Gorenstein algebra of socle degree $e$. If $d > e/2$, then the elements of any component $[I]_d$ do not all have a common factor of any degree $r\geq 1$.
\end{lemma}

\begin{lemma}(\cite[Proposition 3.3]{AADFIMMMN}) \label{AADFIMMMN prop}
Let $A = R/I$ be an artinian Gorenstein algebra  with Hilbert function
\[
h_0 < h_1 < \dots < h_{t-1} < h_t = \dots = h_s > h_{s+1} > h_{s+2} > \dots > h_{d-1} > h_d
\]
where $s \geq t+2$ and $h_s \leq s$. Then $A$ has the WLP.
\end{lemma}
The following consequence of the Snake Lemma will be used extensively. 

\begin{lemma}\label{snake}
    Let $A$ be an artinian Gorenstein algebra and let $f$ be a form of degree $s$ in $A$. Let $B = A/(0:f)$ and $C = A/(f)$. If for some degree $i$ the multiplication maps by a linear form $\ell$ 
    \[
    \times\ell\colon[B]_{i-s}\longrightarrow[B]_{i+1-s}\quad\text{and}\quad
    \times\ell\colon[C]_{i}\longrightarrow[C]_{i+1}
    \]
    are both injective or both surjective, then so is the multiplication map
    \[
    \times\ell\colon[A]_{i}\longrightarrow[A]_{i+1}.
    \] 
\end{lemma}
\begin{proof}
Observe that as a graded $A$-module $(f)\cong B(-s)$, so we have a short exact sequence
\[
0\longrightarrow B(-s)\overset{\times f}\longrightarrow A\longrightarrow C\longrightarrow 0.
\]    
This gives rise to the following commutative diagram:
\begin{equation}\label{eq: CDbis}
      \begin{tikzcd}
        0\arrow{r} & {[B]}_{i-s}\arrow{r}{\times f}\arrow{d}{\times\ell} & {[A]}_{i} \arrow{r} \arrow{d}{\times \ell} & {[C]}_{i} \arrow{r}\arrow{d}{\times \ell} &0\\%
0\arrow{r} &{[B]}_{i+1-s}\arrow{r}{\times f}&{[A]}_{i+1} \arrow{r}& {[C]}_{i+1}\arrow{r} &0
      \end{tikzcd}
\end{equation}
    coming from the short exact sequence.
    By the Snake Lemma, the middle vertical map has maximal rank if  the leftmost and the rightmost vertical maps are either both injective or both surjective.
\end{proof}

\section{Three families satisfying the WLP}\label{S3}

The proof of our main theorem builds on an induction where there are several important base cases. In this section we will establish the WLP for the three families of artinian Gorenstein algebras of odd socle degree $d = 2m+1$  with $h$-vectors 
\begin{itemize}
    \item $(1,3,5,7,9,\dots,2m-1,2m+1,2m+1,2m-1,\dots,9,7,5,3,1)$ (Theorem~\ref{135})
    \item $(1,3,6,8,10,\dots,2m,2m+2,2m+2,2m,\dots,10,8,6,3,1)$ (Theorem~\ref{136})
    \item $(1,4,6,8,10,\dots,2m,2m+2,2m+2,2m,\dots,10,8,6,4,1)$ (Theorem~\ref{keyresult})
\end{itemize}
At the end of the section we will also establish the WLP for artinian Gorenstein algebras of $h$-vector $(1,5,6,6,5,1)$, the SLP for artinian Gorenstein algebras of $h$-vectors $(1,4,4,\dots,4,1)$ and $(1,5,5,\dots,5,1)$,
and for artinian Gorenstein algebras of codimension $4$ and socle degree $4$. These results will be used for the base case in the induction over the socle degree. 
\begin{theorem} \label{135}
    Any artinian Gorenstein algebra $A$ of codimension $3$ of odd socle degree $2m+1$ and $h$-vector 
    \[
    (1, 3, 5, 7, 9, \cdots, 2m-1, 2m+1, 2m+1, 2m-1, \cdots ,9, 7, 5, 3, 1)
    \]
    saytisfies WLP. 
\end{theorem}

\begin{proof}
    Let $R = K[x_1,x_2,x_3]$  and $A = R/I$. The cases where $m\le 2$ are known from before \cite[Corollary 3.12]{BMMNZ}.  There are only two possibilities for the resolution of $A$ over $R$ by the Buchsbaum-Eisenbud structure theorem. Either  the ideal is generated by one quadric and two forms of degree $m+1$ or by one quadric, two forms of degree $m+1$ and two forms of degree $m+2$. In the first case, $I$ is a complete intersection which we know forces WLP by \cite{HMNW}. In the second case, there is a $2\times 3$-submatrix of the $5\times 5$ Buchsbaum-Eisenbud matrix whose maximal minors give the first three generators. Denote this ideal by $J$. By Lemma~\ref{MZ Ill} we have that there is no common factor among the forms in $[I]_{m+1}$. This means that $J$ is a perfect ideal of codimension $2$ and we conclude that the multiplication by a general linear form $\ell$ from  $[A]_m$ to $[A]_{m+1}$ is bijective since it coincides with the multiplication on $R/J$ in these degrees.
\end{proof}

\begin{theorem} \label{136}
    Any artinian Gorenstein algebra $A$ of codimension $3$ of odd socle degree $2m+1$ and $h$-vector 
    \[
    (1, 3, 6, 8, 10, \cdots, 2m, 2m+2, 2m+2, 2m, \cdots ,10, 8, 6, 3, 1)
    \]
    saytisfies WLP. 
\end{theorem}

\begin{proof}
    The cases $m=1,2$ were proven in \cite{BMMNZ}. Consider the case $m=3$. 
We have the $h$-vector $(1,3,6,8,8,6,3,1)$ which is the Hilbert function of a complete intersection of type (3,3,4). There are four possible Betti tables as displayed in Table~\ref{tab:betti_tables}.

\begin{table}[ht]
    \centering
\begin{tabular}{|ccc|ccc|}\hline &&&&&\\
(1)&$\begin{matrix}
\text{total:}&1&3&3&1\\\text{0:}&1&\text{.}&\text{.}&\text{.}\\\text{1:}&\text{.}&\text{.}&\text{.}&\text{.}\\\text{2:}&\text{.}&2&\text{.}&\text{.}\\\text{3:}&\text{.}&1&\text{.}&\text{.}\\\text{4:}&\text{.}&\text{.}&1&\text{.}\\\text{5:}&\text{.}&\text{.}&2&\text{.}\\\text{6:}&\text{.}&\text{.}&\text{.}&\text{.}\\\text{7:}&\text{.}&\text{.}&\text{.}&1\\\end{matrix}$  &&(2)& 
$\begin{matrix}
\text{total:}&1&5&5&1\\\text{0:}&1&\text{.}&\text{.}&\text{.}\\\text{1:}&\text{.}&\text{.}&\text{.}&\text{.}\\\text{2:}&\text{.}&2&\text{.}&\text{.}\\\text{3:}&\text{.}&1&2&\text{.}\\\text{4:}&\text{.}&2&1&\text{.}\\\text{5:}&\text{.}&\text{.}&2&\text{.}\\\text{6:}&\text{.}&\text{.}&\text{.}&\text{.}\\\text{7:}&\text{.}&\text{.}&\text{.}&1\\\end{matrix}$  
&     \\ &&&&&\\\hline &&&&&\\
(3)&$\begin{matrix}     
\text{total:}&1&5&5&1\\\text{0:}&1&\text{.}&\text{.}&\text{.}\\\text{1:}&\text{.}&\text{.}&\text{.}&\text{.}\\\text{2:}&\text{.}&2&1&\text{.}\\\text{3:}&\text{.}&2&\text{.}&\text{.}\\\text{4:}&\text{.}&\text{.}&2&\text{.}\\\text{5:}&\text{.}&1&2&\text{.}\\\text{6:}&\text{.}&\text{.}&\text{.}&\text{.}\\\text{7:}&\text{.}&\text{.}&\text{.}&1\\\end{matrix}$   &&(4)&
$\begin{matrix}
\text{total:}&1&7&7&1\\\text{0:}&1&\text{.}&\text{.}&\text{.}\\\text{1:}&\text{.}&\text{.}&\text{.}&\text{.}\\\text{2:}&\text{.}&2&1&\text{.}\\\text{3:}&\text{.}&2&2&\text{.}\\\text{4:}&\text{.}&2&2&\text{.}\\\text{5:}&\text{.}&1&2&\text{.}\\\text{6:}&\text{.}&\text{.}&\text{.}&\text{.}\\\text{7:}&\text{.}&\text{.}&\text{.}&1\\\end{matrix}$  &\\&&&&&\\\hline
\end{tabular}
\caption{The four possible Betti tables for artinian Gorenstein algebras with $h$-vector $(1,3,6,8,8,6,3,1)$}
\label{tab:betti_tables}
\end{table}

The first is a complete intersection, which we know forces WLP \cite{HMNW}. The third and fourth force the cubic generators to have a common factor, $Q$ {\red{$q$}}, of degree 2 and this will be dealt with in the general argument. For now we focus on the second Betti table.

The Buchsbaum-Eisenbud matrix for this algebra has a submatrix corresponding to the degrees 2 and 3 rows. This corresponds to a map
\[
R(-5)^2 \rightarrow R(-3)^2 \oplus R(-4).
\]
We claim that this submatrix is a Hilbert-Burch matrix, i.e. that the above map gives a free resolution of a zero-dimensional subscheme $Z$ of $\mathbb P^2$ of degree 8. The ideal of maximal minors define a subscheme of $\mathbb P^2$, but Lemma \ref{MZ Ill} guarantees that there is no common factor, so the three minors must define a codimension 2 scheme, and we are done. But this means that through degree 4, the Gorenstein ideal $I$ coincides with the coordinate ring of an algebra with depth 1, so a general linear form gives an isomorphism from degree 3 to degree 4. This settles the case of the second Betti table.

We will return to the third and fourth Betti tables shortly, but we make an observation about the case $m \geq 4$. Now the Hilbert function has the form given in the statement of the proposition. The growth from degree 3 to degree 4 is now maximal according to Macaulay, and in fact it continues to be maximal from any degree $s$ to degree $s+1$ for $s \leq m-1$. Thanks to \cite{BGM} (or Gotzmann), this means that the base locus (as a scheme) of any component of $I$ in degree $\leq m$ has a curve of degree 2 as a component (not necessarily irreducible and not unmixed) with Hilbert polynomial $2t+2$. Since $R$ has three variables, this means that all components of $I$ in degree $\leq m$ have a common factor, $q$, of degree 2. This was true also in the remaining cases for $m=3$, which we now bring back into the picture.

Consider the short exact sequence
\[
0 \rightarrow R/(I:q)(-2) \stackrel{\times q}{\longrightarrow} R/I \rightarrow R/(I,q) \rightarrow 0. 
\]
The cubic generators of $I$ have the form $qx, qy$ after a change of variables, so $I:q$ has two linear forms. Thus the Hilbert function of $R/(I:q)$ is $(1,1,1, \dots)$ and the Hilbert function of $R/(I,q)$ must be
\[
(1,3,5,7,\dots, 2m-1, 2m+1, 2m+1, \dots)
\]
where the two occurrences of $2m+1$ occur in degrees $m$ and $m+1$.

Lemma \ref{snake} gives that for a general linear form $\ell$, the multiplication by $\ell$ from degree $m$ to degree $m+1$ in $R/I$ has maximal rank if and only if the multiplication for $R/(I,q)$ has maximal rank from degree $m$ to degree $m+1$. Suppose this is not the case here and consider the algebra $B = R/(I,q)$. This ideal has one generator (namely $q$) in degree 2 and two in degree $m+1$, say $g_1, g_2$. Let $J = (q,g_1,g_2)$ and consider $R/(J,\ell)$.

Thinking of $q,g_1,g_2$ as being in $K[x_1,x_2,x_3,x_4]$, let $\ell_1$ be a general linear form (which without loss of generality we can think of as being $x_4$). We have that $$R/(J,\ell) \cong  K[x_1,x_2,x_3,x_4]/(J,\ell_1,\ell_2). $$ From the exact sequence
\[
0 \rightarrow R/(J:\ell)(-1) \stackrel{\times \ell}{\longrightarrow} R/J \rightarrow R/(J,\ell) \rightarrow 0
\]
failure of maximal rank gives that $\dim [R/(J,\ell)]_{m+1} \geq 1$ and in fact it must be exactly 1. Letting $a = 2 = \deg q$, Proposition \ref{MNZ result} says that $q, g_1, g_2$ have a common factor of degree 1. By our construction, this common factor must also be present in all components of $I$ of degree $\leq m+1$. But $d = m+1 > e/2  = m+1/2$, so by Lemma \ref{MZ Ill} we get a contradiction.

This takes care of the third and fourth Betti tables for $m=3$ and also the case $m \geq 4$, so we are finished with the proof.
\end{proof}

Before proving the WLP for artinian Gorenstein algebras of the third family, we need to study the possible four-dimensional subspaces of quadrics in four variables that can be the degree two part of the ideals.

\begin{theorem}
    \label{thm:quadrics}
Consider an ideal $I \subset R = K[x_1,\ldots,x_4]$ that is minimally generated by four quadrics. If the generic initial ideal of $I$ with respect to the lexicographic order contains $x_1^2,x_1x_2,x_1x_3,x_2^2$, then up to a change of coordinates,  one of the following must be true:
\begin{itemize}
    \item[(i)] $I =(x_1x_3,x_1x_4,x_2x_3,x_2x_4)$;\; 
     \item[(ii)] $I = (x_1^2,x_2^2,x_3^2,x_1x_2+x_1x_3+x_2x_3)$;
          \item[(iii)] $I = (x_1^2,x_2^2,x_3^2,x_1x_3+x_2x_3)$;
     \item[(iv)] $I = (x_1^2,x_1x_2,x_1x_3-x_2^2,x_3^2)$;
     \item[(v)] $I = (x_1^2,x_1x_2,x_1x_3-x_2^2,x_2x_3)$;
     \item[(vi)] $I = (x_1^2,x_1x_3,x_2^2,x_2x_3)$;     
     \item [(vii)] $I = (x_1^2,x_1x_2,x_1x_3,x_2x_3)$;
     \item[(viii)] $I = (x_1^2,x_1x_2,x_2^2,q)$, where  $q=x_3x_4$, $q=x_3^2+x_2x_4$ or $q = x_3^2$.  
     \item[(ix)] $I = (x_1^2,x_1x_2,x_2^2,x_1x_4-x_2x_3)$; \; or
    \item[(x)] $I = (x_1^2,x_1x_2,x_2^2,x_1x_3)$. 
    \end{itemize}
Furthermore, the Hilbert function of $I$ is
\begin{itemize}
    \item $(1,4,6,6,6,\dots)$ in cases $(ii)$, $(iii)$, $(iv)$ and $(viii)$,
    \item $(1,4,6,7,8,9,\dots)$ in cases $(v)$ and $(vi)$, and 
    \item $(1,4,6,8,10,\dots)$ in cases $(i)$, $(vii)$, $(ix)$ and $(x)$. 
\end{itemize}
\end{theorem}

\begin{proof} 
    By the definition of the lexicographic order, any quadric whose initial monomial with respect to the lexicographic order is $x_2^2$ is in $K[x_2, x_3, x_4]$. Hence, 
    the assumption on the generic initial ideal means that up to scaling there is a unique quadratic form in $I$ that does not involve $x_1$ after a general change of coordinates. This implies that for a general choice of $H=\langle\ell_1,\ell_2,\ell_3\rangle\subseteq [R]_1$,  there is a unique quadratic form $q_H\in H^2\cap [I]_2$ up to scaling.

    This gives a rational map 
    \[
    \begin{array}{cccccccc}
    \Phi\colon  \mathbb P^3 & = & \mathbb P([R]_1) & \longrightarrow &  \mathbb P^3 & = & \mathbb P([I]_2). \\
    && H & \mapsto & q_H
    \end{array}
    \]
    By \Cref{hesscriterium}, the locus of quadrics in $\mathbb P([I]_2)$ of rank less than $4$ is given by a quartic polynomial $f$ which is the determinant of the Hessian matrix of the form $a_1q_1+a_2q_2+a_3q_3+a_4q_4$, where $a_1,a_2,a_3,a_4$ are the homogeneous coordinates for  $\mathbb P([I]_2)$. Since the rank of $q_H$ is at most three, the image of $\Phi$ is contained in $V(f)$.

    We now divide the proof into three cases depending on the dimension of the image of $\Phi$. These cases correspond to the rank of a general $q_H$ in the image of $\Phi$. If  $\dim \operatorname{im}(\Phi) = 3$, the map is dominant and the rank of $q_H$ is generically $3$. If the $\dim \operatorname{im}(\Phi) = 2$ the fibers are one-dimensional which corresponds to the rank of $q_H$ to be generically $2$. Finally, if $\dim\operatorname{im}(\Phi)=1$, we have two-dimensional fibers and the generic rank of $q_H$ is one.
     
    We first assume that $\dim\operatorname{im}(\Phi)=3$ so that the rank of $q_H$ is generically $3$. Then the map $\Phi$ is dominant and one-one since we can recover $H=(\ell_1,\ell_2,\ell_3)$ from  $q_H(\ell_1,\ell_2,\ell_3)$. By Bertini's Theorem \cite{K}, a general element of the linear system $[I]_2$ is smooth away from the base locus. This gives a contradiction since the singular point of $q_H$ varies over an open set of $\mathbb P^3$ as $H$ varies. We conclude that the rank of $q_H$ cannot be generically $3$. 

    We now assume that $\dim\operatorname{im}(\Phi) = 2$ which means that the rank of generic $q_H$ is $2$. Thus the closure of  $\operatorname{im}(\Phi)$ is a surface $V\subseteq \mathbb P([I]_2)$. 

    If $V=\langle q_1,q_2,q_3\rangle$ is a plane,  we have that the general element of the linear system of this plane is smooth away from the base locus. Since the quadrics of rank two  have a singular line, these singular lines must be contained in the base locus. Either the singular lines are all the same, which means that $\langle q_1,q_2,q_3\rangle = \langle x_1^2,x_1x_2,x_2^2\rangle$ after a change of coordinates, or the base locus contains a plane which means that  $q_1,q_2,q_3$ have a common factor $\ell$. However, we cannot have that $q_H$ is divisible by $\ell$ for all $H$, since  $(x_2,x_3,x_4)^2$ contains no forms divisible by $x_1$.

        If $\langle q_1,q_2,q_3\rangle = \langle x_1^2,x_1x_2,x_2^2\rangle$, we have that
    $q_H=\ell^2$ is the square of the unique linear form in $\ell\in H\cap \langle x_1,x_2\rangle$ for general choice of $H$. Hence the generic rank of $q_H$ is one, contradicting the assumption that $\dim\operatorname{im}(\Phi)=2$.

 We now go to the case where $V$ is not contained in a plane. 
    A general line in $\mathbb P([I]_2)$ will be a pencil of quadrics where the general element has rank $3$ or $4$. We start with the case where the general element has rank $4$. Since $V$ is not contained in a plane, there are at least two points of rank $2$ on the line. A pencil of quadrics spanned by two rank $2$ quadrics where the general element has rank $4$ has to be of the form $\langle \ell_1\ell_2,\ell_3\ell_4\rangle$, where $\{\ell_1,\ell_2,\ell_3,\ell_4\}$ is linearly independent. Hence there are exactly two points of intersection between the line and $V$ and we conclude that $V$ is a quadric surface. This surface is either smooth or a cone over a smooth conic. In any case, through two general points on $V$, we can always find lines in $V$ that meet at a third point. The lines through  $\langle \ell_1\ell_2\rangle$ that can be contained in $V$ are lines sharing a common factor $\ell_1$ or $\ell_2$ and lines in the plane $\langle \ell_1^2,\ell_1\ell_2,\ell_2^2\rangle$. The lines in $\langle \ell_1^2,\ell_1\ell_2,\ell_2^2\rangle$ do not meet any of the lines through $\langle\ell_3\ell_4\rangle$ and the lines in $\langle \ell_3^2,\ell_3\ell_4,\ell_4^2\rangle$ do not meet any of the lines through $\langle\ell_1\ell_2\rangle$.
    If $V$ were a cone, the apex would have to be one of the four quadrics $\ell_1\ell_3$, $\ell_1\ell_4$, $\ell_2\ell_3$ or $\ell_2\ell_4$. We get a contradiction by replacing  $\langle\ell_1\ell_2\rangle$ and $\langle\ell_3\ell_4\rangle$ by new general points $\langle\ell_1'\ell_2'\rangle$ and $\langle\ell_3'\ell_4'\rangle$ unless all quadrics on $V$ share a common factor, but this is impossible since $\ell_1\ell_2$ and $\ell_3\ell_4$ have no common factor. Hence we conclude that $V$ is a smooth quadric surface with two rulings. Since all the lines in one ruling meet all the lines in the other ruling,  the only possibilities are 
    \[
     [I]_2 = \langle\ell_1\ell_2, \ell_3\ell_4, \ell_1\ell_3, \ell_2\ell_4\rangle 
    \quad\text{or}\quad
    [I]_2 = \langle\ell_1\ell_2, \ell_3\ell_4, \ell_1\ell_4, \ell_2\ell_3\rangle
    \]
    and after a change of variables this shows that we are in case $(i)$.
    
    Now we look at the case where the general element of $\mathbb P([I]_2)$ has rank $3$. 
    Given one quadric $q_1$ of rank $3$, we can change coordinates so that it becomes $q_1 = x_1^2+x_2^2+x_3^2$. Consider the quadrics $q_2$ such that the general member of the pencil $\langle q_1,q_2\rangle$ has rank $3$. These forms correspond to symmetric matrices such that 
    \[
    p(\lambda)=\det \left(\lambda\begin{bmatrix}
        1&0&0&0\\0&1&0&0\\0&0&1&0\\0&0&0&0
    \end{bmatrix}+\begin{bmatrix}
        a_1&a_2&a_4&a_7\\
        a_2&a_3&a_5&a_8\\
        a_4&a_5&a_6&a_9\\
        a_7&a_8&a_9&a_{10}
    \end{bmatrix}\right) = 0 ,\quad \forall \lambda\in K.
    \]
    Let $Z$ be the variety in $\mathbb A^{10}$ given by the ideal generated by the coefficients of $p(\lambda)$ in $K[a_1,a_2,\dots,a_{10}]$. This is a variety with two components of dimension $6$.
    The coefficients of $\lambda^3$ and $\lambda^2$ in $p(\lambda)$ are $a_{10}$ and $a_7^2+a_8^2+a_9^2$. If $a_7=a_8=a_9=a_{10} = 0$ then $p(\lambda)=0$,  and we have that $q_2$ is a quadric in the subring $K[x_1,x_2,x_3]$. This is one of the components. Otherwise, we have at least one of $a_7$, $a_8$ and $a_9$ non-zero and we have that the other component of $Z$ is fibered over the smooth conic $a_7^2+a_8^2+a_9^2=0$ in the plane with projective coordinates $(a_7:a_8:a_9)$. The linear group fixing this curve acts transitively on the points of the curve and the curve does not contain any line. Hence, if we look for a linear space inside this component of the variety $Z$, we can change coordinates so that $a_9=a_{10}=0$, and $a_8=ia_7$. Then the coefficient for $\lambda$ is $(a_1+2ia_2-a_3)a_7$ and the constant term 
    \[
    a_7^2(a_3a_6-a_5^2+2ia_4a_5-2ia_2a_6+a_4^2-a_1a_6 )=a_6a_7^2(a_1+2ia_2-a_3) + a_7^2(a_4+ia_5)^2.
    \]
    Hence we get the five-dimensional linear space given by the ideal $(a_{10},a_9,a_8-ia_7,a_1+2ia_2-a_3,a_4+ia_5)$. Written as a linear system of quadrics, this is 
    \[
    \langle
    x_1^2+x_2^2,x_1x_2-ix_1^2,x_1x_3+ix_2x_3,x_1x_4+ix_2x_4,x_3^2\rangle.
    \]
    A general member of this space is a quadric of rank $3$. After yet another change of coordinates, we can write this linear system as 
    \[
    W = \langle
    x_1^2,x_1x_2,x_1x_3,x_1x_4,x_2^2
    \rangle
    \]
    and the locus of quadrics of rank $2$ in $W$ is given by the ideal of $3$-minors of 
    \[
    \begin{bmatrix}
        b_1&b_2&b_3&b_4\\
        b_2&b_5&0&0\\
        b_3&0&0&0\\
        b_4&0&0&0
    \end{bmatrix}
    \]
    whose radical is $(b_3b_5,b_4b_5) = (b_3,b_4)\cap(b_5)$. Hence this locus is 
    contained in the union
    \[
    \langle x_1^2,x_1x_2,x_1x_3,x_1x_4\rangle
    \cup 
    \langle
    x_1^2,x_1x_2,x_2^2
    \rangle.
    \]
    We have assumed that $V$ is not contained in a plane and that $V$ is contained in the rank $2$ locus. Hence the only possibility is  that $I = (x_1^2,x_1x_2,x_1x_3,x_1x_4)$. Howver, this contradicts the assumption on the initial ideal. We conclude that the only possibility for the rank of a general element of $\mathbb P([I]_2)$ to be $3$ is that $a_7=a_8=a_9=a_{10}=0$ and $[I]_2\subseteq \langle x_1^2,x_1x_2,x_1x_3,x_2^2,x_2x_3,x_3^2\rangle$ after a change of variables. Consider the ideal $J = S\cap I$, where $S:=K[x_1,x_2,x_3]\subseteq R=K[x_1,x_2,x_3,x_4]$.
    The ideal $J$ has codimension $2$ or $3$. If the codimension is $3$, $[J]_2$ contains a regular sequence  of length $3$ and since we cannot have socle in degree $2$ in this complete intersection, we conclude that the Hilbert function of $J$ is $(1,3,2)$. If the codimension is $2$, the degree can  either be degree $1$ or $2$. Hence there are the following three possible Hilbert functions for $J$ under our assumptions. 
    \begin{enumerate}
        \item $(1,3,2)$
        \item $(1,3,2,1,1,1,\dots)$
        \item $(1,3,2,2,2,2,\dots)$        
    \end{enumerate}
    We deal with these three cases one by one.
    \begin{enumerate}
        \item In this case of Hilbert function $(1,3,2)$, the inverse system of $J$ in degree two is given by a pencil of quadrics where no element has rank $1$. If there is a common factor between the two generators, we have $\langle X_1X_3,X_2X_3\rangle$ after a change of variables and $I = (x_1^2,x_1x_2,x_2^2,x_3^2)$ which is a special case of $(viii)$. If the two generators form a complete intersection, the base locus can  be four points, three points or two points, since one point would force a square of a linear form in the pencil. If the base locus is four points, we have $\langle X_1(X_2-X_3),X_2(X_1-X_3)\rangle$ after a change of coordinates and we get $I = (x_1^2,x_2^2,x_3^2,x_1x_2+x_1x_3+x_2x_3)$ as in case $(ii)$. If the base locus is three points, we have $\langle X_1X_2,X_3(X_1-X_2)\rangle$ after a change of coordinates and we get $I = (x_1^2,x_2^2,x_3^2,x_1x_3+x_2x_3)$ as in case $(iii)$. If the base locus is two points, the multiplicities have to be three and one, since if it was two of multiplicity two, the square of the line through the points would be in the pencil. A pencil of quadrics of rank two either has a common factor or is in two variables which forces a square in the pencil. Hence the general element is smooth and we can change coordinates such that $Q_1 = X_2^2+X_1X_3$ and $Q_2 = X_2X_3$, which gives $I = (x_1^2,x_1x_2,x_1x_3-x_2^2,x_3^2)$ as in case $(iv)$.
        
        \item If the Hilbert function if $(1,3,2,1,1,1,\dots)$, there are two possibilities up to change of coordinates. The inverse system in degree two is a pencil of quadrics. 
        Since the inverse system in degree three and higher is given by powers of a linear form, $L$, we have that $I^\perp_2$ contains a square of a linear form. By a change of variables we can assume that this is $X_3^2$. Hence $[I^\perp]_2 = \langle Q,X_3^2\rangle$. There cannot be socle in degree one, since there are no generators of degree three. Hence $[I^\perp]_2$ generates $[I^\perp]_1$.  
        If the pencil contains a quadric of rank $2$, we can assume that $Q = L_1L_2$ for two linear forms. Since there is no socle in degree one, we have that $\dim_K \langle L_1,L_2,X_3\rangle = 3$ and we can assume that $Q = X_1X_2$. This gives $I = (x_1^2,x_1x_3,x_2^2,x_2x_3)$ as in $(vi)$.

        If the pencil does not have any quadric of rank two, we have that $Q+\lambda X_3^2$ has rank $3$ for all $\lambda\in K$. Let 
        \[
        \begin{bmatrix}
            c_1&c_2&c_4\\
            c_2&c_3&c_5\\
            c_4&c_5&c_6
        \end{bmatrix}
        \]
        be the symmetric matrix representing $Q$. Then this condition on $Q$ means that 
        \[
        \det \begin{bmatrix}
            c_1&c_2\\c_2&c_3
        \end{bmatrix}=0
        \]
        and hence we can assume that $Q = X_2^2 + (\alpha X_1+\beta X_2)X_3$ after a change of basis. Since there is no socle in degree one, we cannot have $\alpha=0$, so we can assume that $Q = X_2^2+X_1X_3$ after a change of coordinates. This gives $I = (x_1^2,x_1x_2,x_1x_3-x_2^2,x_2x_3)$ as in $(v)$.
        \item When the Hilbert function is $(1,3,2,2,2,\dots)$, we have that the inverse system in degree $3$ is either generated by two cubes or by $L_1^3$ and $L_1^2L_2$ for some linear forms. This gives the two possibilities up to change of variables. Either we have $I = (x_1^2,x_1x_2,x_1x_3,x_2x_3)$ as in $(vii)$ or $I = (x_1^2,x_1x_2,x_1x_3,x_2^2)$ as in $(x)$.
    \end{enumerate}

    At last, we assume that $\dim\operatorname{im}(\Phi)=1$ which means that the rank of $q_H$ is one for a general $H$ and the closure of $\operatorname{im}(\Phi)$ is a curve. Look at $W = \langle \ell\colon \ell^2=q_H \text{ for some } H\rangle$. Since $\operatorname{im}(\Phi)$ is a curve, $\dim_K W\ge 2$. If $\dim_K W = 4$, $I$ contains a complete intersection of quadrics, which contradicts the assumption on the initial ideal. If $\dim_K W = 3$, we can assume that $[I]_2 = \langle x_1^2,x_2^2,x_3^2,q\rangle$ for some quadric $q = a_1x_1x_2+a_2x_1x_3+a_3x_1x_4+a_4x_2x_3+a_5x_2x_4+a_6x_3x_4+a_7x_4^2$. With a change of coordinates $(x_1,x_2,x_3,x_4)\mapsto(x_1,x_2+\alpha x_1,x_3+\beta x_1,x_4+\gamma x_1)$ we get that the assumption on the generic initial ideal gives the condition that 
    \[
    4\alpha\beta (a_3+\alpha a_5 + \beta a_6+2\gamma a_7) = 0
    \]
    for general choice of $\alpha,\beta,\gamma$ in $K$. Hence $a_3=a_5=a_6=a_7=0$ showing that 
    $q = a_1x_1x_2+a_2x_1x_3+a_4x_2x_3$ and that we are in the case (1) above.

    Thus we conclude that $\dim_K W = 2$ and $[I]_2 =\langle x_1^2,x_1x_2,x_2^2,q\rangle$ for some quadric $q$ after a change of variables.  We have two cases, either $q\in (x_1,x_2)$ or $q\notin(x_1,x_2)$. If $q\in (x_1,x_2)$ we can assume after a change of coordinates that $q = x_1x_4-x_2x_3$ which is case $(ix)$ or $q = x_1x_3$ which is case $(x)$. If $q\notin (x_1,x_2)$, we can write $q = q_1(x_3,x_4)+(\alpha x_1+\beta x_2)x_3+(\gamma x_1+\delta x_2)x_4$, where $q_1\ne 0$. If $q_1$ has rank $1$, we can assume that $q_1=x_3^2$ and we can write $q = (x_3+\alpha/2x_1+\beta/2x_2)^2 + (\gamma x_1+\delta x_2)x_4 -(\alpha x_1+\beta x_2)^2/4$ which is $q = x_3^2$ or $q = x_3^2+x_2x_4$ after a change of variables depending on whether $\gamma=\delta=0$ or not. If $q_1$ has rank $2$, we can assume that $q_1 = x_3x_4$ and we can write $q = (x_3+\gamma x_1+\delta x_2)(x_4+\alpha x_1+\beta x_2) - (\gamma x_1+\delta x_2)(\alpha x_1+\beta x_2)^2$, so we can assume that $q = x_3x_4$ after a change of coordinates. Together, this is case $(viii)$.
\end{proof}

We are now ready to prove that the WLP holds in the third of our families using Proposition~\ref{exactsequence} and Theorem~\ref{136}.
    
\begin{theorem} \label{keyresult}
    Any artinian Gorenstein algebra $A$ of odd socle degree odd, $d=2m+1$, and $h$-vector 
\begin{equation}\label{1468}
    (1, 4, 6, 8, 10, 2m, \cdots, 2m+2, 2m+2, 2m, \cdots ,10, 8, 6, 4, 1)
\end{equation}
satisfies the WLP. 
\end{theorem}

\begin{proof}
It is enough to consider the multiplication $\times \ell\colon [A]_m\longrightarrow [A]_{m+1}$.
    Let $A = R/I$ be an artinian Gorenstein algebra with $h$-vector (\ref{1468}) where $R = K[x_1,x_2,x_3,x_4]$.
    
    Then the quadratic part of the ideal,  $[I]_2$, is four-dimensional. There are two possibilities for the generic initial monomials with respect to the \emph{lexicographic} order:
    \[
    \{x_1^2,x_1x_2,x_1x_3,x_1x_4\} \quad\text{and}\quad
    \{x_1^2,x_1x_2,x_1x_3,x_2^2\}
    \]
    since these are the only stable sets of four monomials under the Borel action.
    We will divide the proof of the theorem according to these two cases. Observe that we use the lexicographic order rather than the reverse lexicographic order. This means that the first case is the generic while the other case is special.
    
    In the first case, after a general change of coordinates, the leading monomials of the forms in $[I]_2$  are $\{x_1^2,x_1x_2,x_1x_3,x_1x_4\}$. This implies that the subalgebra $B$ generated by $x_2,x_3,x_4$ in $A$ that shares the Hilbert function with $A$ except for in degree $1$. In this case the pairing in the middle and the corresponding Hessian will be the same as in a Gorenstein artin algebra in three variables with $h$-vector 
    \[
    (1,3,6,8,10,\dots,2m+2,2m+2,\dots,10,8,6,3,1).
    \] 
    and we know from Proposition~\ref{136} that this has the WLP. 
    
    In the second case, we have that after a general change of coordinates $[I]_2$ has initial monomials $\{ x_1^2,x_1x_2,x_1x_3,x_2^2\}$. We can now apply Theorem~\ref{thm:quadrics} and we will consider the ten  cases
    \begin{itemize}
    \item[$(i)$] $[I]_2 =\langle x_1x_3,x_1x_4,x_2x_3,x_2x_4\rangle$;\; 
     \item[$(ii)$] $[I]_2 = \langle x_1^2,x_2^2,x_3^2,x_1x_2+x_1x_3+x_2x_3\rangle$;
     \item[$(iii)$] $[I]_2 = \langle x_1^2,x_2^2,x_3^2,x_1x_3+x_2x_3\rangle$;
     \item[$(iv)$] $[I]_2 = \langle x_1^2,x_1x_2,x_1x_3-x_2^2,x_3^2\rangle$;
     \item[$(v)$] $[I]_2 = \langle x_1^2,x_1x_2,x_1x_3-x_2^2,x_2x_3\rangle$;
     \item[$(vi)$] $[I]_2 = \langle x_1^2,x_1x_3,x_2^2,x_2x_3\rangle$;
     \item [$(vii)$] $[I]_2 = \langle x_1^2,x_1x_2,x_1x_3,x_2x_3\rangle$;
     \item[$(viii)$] $[I]_2 = \langle x_1^2,x_1x_2,x_2^2,q\rangle$, where $q\notin (x_1,x_2)$.       
     \item[$(viii)$] $[I]_2 = \langle x_1^2,x_1x_2,x_2^2,x_1x_4-x_2x_3\rangle$; \; or
    \item[$(x)$] $[I]_2 = \langle x_1^2,x_1x_2,x_2^2,x_1x_3\rangle$. 
    \end{itemize}

\noindent $(i)$
    When  $[I]_2 = \langle x_1x_3,x_1x_4,x_2x_3,x_2x_4\rangle$ we have that the Gorenstein algebra $A = R/I$ is the connected sum of one Gorenstein algebra in $k[x_1,x_2]$ and one in $k[x_3,x_4]$. Hence it satisfies the SLP.
    
    \noindent $(ii)$, $(iii)$, $(iv)$ and $(viii)$
    When $[I]_2 = \langle x_1^2,x_2^2,x_3^2,x_1x_2+x_1x_3+x_2x_3\rangle$, $[I]_2 = \langle x_1^2,x_2^2,x_3^2,x_1x_3+x_2x_3\rangle$ or $[I]_2 = \langle x_1^2,x_1x_2,x_2^2,q\rangle$, where $q\notin (x_1,x_2)$, the Hilbert function of $A = R/I$ is bounded from above by the Hilbert function of $R/([I]_2)$ which is $(1,4,6,6,6,\dots)$.  Hence these cases can only occur for $m=2$. Here $([I]_2)$ is the ideal of a zero-dimensional scheme in $\mathbb P^3$ and a non-zero-divisor $\ell$ gives injective multiplication on $A = R/I$ up to degree $3$. We conclude that $A$ has the WLP.
    
\noindent $(v)$ 
    When $[I]_2 = \langle x_1^2,x_1x_2,x_1x_3-x_2^2,x_2x_3\rangle$ the Hilbert function of $A=R/I$ is bounded by $(1,4,6,7,8,9,\dots)$ and we are again forced to have $m=2$. A general element of the inverse system of $([I]_2)$ in degree five can be written as 
       \[
       \begin{split}
       F =  
       \frac{a_1}{2!3!}X_{2}^{2}X_{4}^{3}
       +\frac{a_1}{3!}X_{1}X_{3}X_{4}^{3}       
       +\frac{a_2}{4!}X_{1}X_{4}^{4}
       +\frac{a_3}{4!}X_{2}X_{4}^{4}
       +\frac{a_4}{5!}X_{3}^{5}
       \\+\frac{a_5}{4!}X_{3}^{4}X_{4}
       +\frac{a_6}{2!3!}X_{3}^{3}X_{4}^{2}
        +\frac{a_7}{2!3!}X_{3}^{2}X_{4}^{3}
       +\frac{a_8}{4!}X_{3}X_{4}^{4}
       +\frac{a_9}{5!}X_{4}^{5}
      \end{split}
       \]
       and with the monomial bases for $[R/([I]_2)]_2$ and $[R/([I]_2)]_3$ given by 
       \[
       \left\{{x}_{1}{x}_{3},\,{x}_{1}{x}_{4},\,{x}_{2}{x}_{4},\,{x}_{3}^{2},\,{x}_{3}{x}_{4},\,{x
      }_{4}^{2}\right\}
       \]
       and 
       \[
       \left\{{x}_{1}{x}_{3}{x}_{4},\,{x}_{1}{x}_{4}^{2},\,{x}_{2}{x}_{4}^{2},\,{x}_{3}^{3},\,{x
      }_{3}^{2}{x}_{4},\,{x}_{3}{x}_{4}^{2},\,{x}_{4}^{3}\right\}
       \]
       we get the catalecticant matrix
       \[
       \begin{bmatrix}
      0&0&0&0&0&0&{a}_{1}\\
      0&0&0&0&0&{a}_{1}&{a}_{2}\\
      0&0&{a}_{1}&0&0&0&{a}_{3}\\
      0&0&0&{a}_{4}&{a}_{5}&{a}_{6}&{a}_{7}\\
      0&{a}_{1}&0&{a}_{5}&{a}_{6}&{a}_{7}&{a}_{8}\\
      {a}_{1}&{a}_{2}&{a}_{3}&{a}_{6}&{a}_{7}&{a}_{8}&{a}_{9}\end{bmatrix}
       \]
       and the second Hessian matrix restricted to $x_3$ and $x_4$ is given by 
              \[
        \begin{bmatrix}
      0&0&0&0&0&{a}_{1}{x}_{4}\\
      0&0&0&0&{a}_{1}{x}_{4}&{a}_{1}{x}_{3}+{a}_{2}{x}_{4}\\
      0&0&{a}_{1}{x}_{4}&0&0&{a}_{3}{x}_{4}\\
      0&0&0&{a}_{4}{x}_{3}+{a}_{5}{x}_{4}&{a}_{5}{x}_{3}+{a}_{6}{x}_{4}&{a}_{6}{x}_{3}+{a}_{7}{x}_{4}\\
      0&{a}_{1}{x}_{4}&0&{a}_{5}{x}_{3}+{a}_{6}{x}_{4}&{a}_{6}{x}_{3}+{a}_{7}{x}_{4}&{a}_{7}{x}_{3}+{a}_{8}{x}_{4}\\
      {a}_{1}{x}_{4}&{a}_{1}{x}_{3}+{a}_{2}{x}_{4}&{a}_{3}{x}_{4}&{a}_{6}{x}_{3}+{a}_{7}{x}_{4}&{a}_{7}{x}_{3}+{a}_{8}{x}_{4}&{a}_{8}{x}_{3}+{a}_{9}{x}_{4}\end{bmatrix}
       \]
       The determinant of this matrix is 
       \[
       (a_1x_4)^5(a_4x_3+a_5x_4).
       \]
       Since the catalecticant matrix has rank $6$, we have that $a_1\neq 0$ and that the matrix 
       \[
       \begin{bmatrix}
           a_4&a_5\\a_5&a_6
       \end{bmatrix}
       \]
       has rank $2$. Hence the determinant of the second Hessian cannot vanish completely and $A$ satisfies the WLP.
       
     \noindent $(vi)$ When $[I]_2 =\langle x_1^2,x_1x_3,x_2^2,x_2x_3\rangle$ the Hilbert function of $A = R/I$ is bounded by $(1,4,6,7,8,9,\dots)$ and we are forced to have $m=2$.  We can write the general element of the inverse system of $([I]_2)$ in degree five as 
     \[\begin{split}
     F =  
     \frac{a_1}{3!}X_{1}X_{2}X_{4}^{3}
     +\frac{a_2}{4!}X_{1}X_{4}^{4}
     +\frac{a_3}{4!}X_{2}X_{4}^{4}
     +\frac{a_4}{5!}X_{3}^{5}
     +\frac{a_5}{4!}X_{3}^{4}X_{4}
     \\+\frac{a_6}{2!3!}X_{3}^{3}X_{4}^{2}
     +\frac{a_7}{2!3!}X_{3}^{2}X_{4}^{3}
     +\frac{a_8}{4!}X_{3}X_{4}^{4}
     +\frac{a_9}{5!}X_{4}^{5}         
     \end{split}
     \]
     When using 
     \[
     \left\{{x}_{1}{x}_{2},\,{x}_{1}{x}_{4},\,{x}_{2}{x}_{4},\,{x}_{3}^{2},\,{x}_{3}{x}_{4},\,{x}_{4}^{2}\right\}
     \,\text{and}\,
     \left\{{x}_{1}{x}_{2}{x}_{4},\,{x}_{1}{x}_{4}^{2},\,{x}_{2}{x}_{4}^{2},\,{x}_{3}^{3},\,{x}_{3}^{2}{x}_{4},\,{x}_{3}{x}_{4}^{2},\,{x}_{4}^{3}\right\}
     \]
     as monomial bases for $R/([I]_2)$ in degree $2$ and $3$, we get the catalecticant matrix
     \[
      \begin{bmatrix}
      0&0&0&0&0&0&{a}_{1}\\
      0&0&{a}_{1}&0&0&0&{a}_{2}\\
      0&{a}_{1}&0&0&0&0&{a}_{3}\\
      0&0&0&{a}_{4}&{a}_{5}&{a}_{6}&{a}_{7}\\
      0&0&0&{a}_{5}&{a}_{6}&{a}_{7}&{a}_{8}\\
      {a}_{1}&{a}_{2}&{a}_{3}&{a}_{6}&{a}_{7}&{a}_{8}&{a}_{9}\end{bmatrix} .
     \]
     Since this has rank $6$, we have that $a_1\neq 0$ and that the matrix
     \[
            \begin{bmatrix}
           a_4&a_5&a_6\\a_5&a_6&a_7\\
       \end{bmatrix}
     \]
     has rank $2$. 
     The restriction of the second Hessian matrix to $x_3$ and $x_4$ is given by 
     \[
     \begin{bmatrix}
          0&0&0&0&0&{a}_{1}{x}_{4}\\
      0&0&{a}_{1}{x}_{4}&0&0&{a}_{2}{x}_{4}\\
      0&{a}_{1}{x}_{4}&0&0&0&{a}_{3}{x}_{4}\\
      0&0&0&{a}_{4}{x}_{3}+{a}_{5}{x}_{4}&{a}_{5}{x}_{3}+{a}_{6}{x}_{4}&{a}_{6}{x}_{3}+{a}_{7}{x}_{4}\\
      0&0&0&{a}_{5}{x}_{3}+{a}_{6}{x}_{4}&{a}_{6}{x}_{3}+{a}_{7}{x}_{4}&{a}_{7}{x}_{3}+{a}_{8}{x}_{4}\\
      {a}_{1}{x}_{4}&{a}_{2}{x}_{4}&{a}_{3}{x}_{4}&{a}_{6}{x}_{3}+{a}_{7}{x}_{4}&{a}_{7}{x}_{3}+{a}_{8}{x}_{4}&{a}_{8}{x}_{3}+{a}_{9}{x}_{4}
     \end{bmatrix}
     \]
     which has determinant 
     \[
\left({{a}_{1}}{{x}_{4}}\right)^{4}\left(\left({a}_{4}{a}_{6}-{a}_{5}^{2}\right){x}_{3}^{2}+\left({a}_{4}{a}_{7}-{a}_{5}{a}_{6}\right){x}_{3}{x}_{4}+\left({a}_{5}{a}_{7}-{a}_{6}^{2}\right){x}_{4}^{2}\right).
     \]
     Hence the condition on the catalecticant forces the determinant of the second Hessian matrix to be non-vanishing and we conclude that $A=R/I$ satisfies the WLP.
     
     \noindent $(vii)$ When $[I]_2 = \langle x_1^2,x_1x_2,x_1
     x_3,x_2x_3\rangle$ we have that the inverse system of $([I]_2)$ in degree $d$ is given by 
     \[
     \left\{
     \frac{X_1X_4^{d-1}}{(d-1)!},
     \frac{X_2^d}{d!},\frac{X_2^{d-1}X_4}{(d-1)!},\dots,\frac{X_2X_4^{d-1}}{(d-1)!},
     \frac{X_3^d}{d!},\frac{X_3^{d-1}X_4}{(d-1)!},\dots,\frac{X_3X_4^{d-1}}{(d-1)!},\frac{X_4^d}{d!}    
     \right\}
     \]
     and we can use the corresponding monomial bases without factorials for $[R/([I]_2)]_m$ and $[R/([I]_2)]_{m+1}$ to get the catalecticant matrix
     \[
     \begin{bmatrix}
         0&0&0&\cdots&0&0&0&\cdots&0&a_1\\
         0&a_2&a_3&\cdots&a_{m+2}&0&0&\cdots&0&a_{m+3}\\
         0&a_3&a_4&\cdots&a_{m+3}&0&0&\cdots&0&a_{m+4}\\
         \vdots&\vdots&\vdots&\ddots&\vdots&\vdots&\vdots&\ddots&\vdots&\vdots\\
         0&a_{m+1}&a_{m+2}&\cdots&a_{2m+1}&0&0&\cdots&0&a_{2m+2}\\     0&0&0&\cdots&0&a_{2m+3}&a_{2m+4}&\cdots&a_{3m+3}&a_{3m+4}\\
0&0&0&\cdots&0&a_{2m+4}&a_{2m+5}&\cdots&a_{3m+4}&a_{3m+5}\\
 \vdots&\vdots&\vdots&\ddots&\vdots&\vdots&\vdots&\ddots&\vdots&\vdots\\
0&0&0&\cdots&0&a_{3m+2}&a_{3m+3}&\cdots&a_{4m+2}&a_{4m+3}\\
a_1&a_{m+2}&a_{m+3}&\cdots&a_{2m+2}&a_{3m+3}&a_{3m+4}&\cdots&a_{4m+3}&a_{4m+4}\\
     \end{bmatrix}
     \]
     Hence $a_1\ne 0$ and both matrices 
     \[
     \begin{bmatrix}
         a_2&a_3&\cdots&a_{m+2}\\
         a_3&a_3&\cdots&a_{m+2}\\
\vdots&\vdots&\ddots&\vdots\\         a_{m+1}&a_{m+2}&\cdots&a_{2m+1}\\
     \end{bmatrix}\text{ and }
          \begin{bmatrix}
         a_{2m+3}&a_{2m+4}&\cdots&a_{3m+3}\\
         a_{2m+4}&a_{2m+5}&\cdots&a_{3m+4}\\
\vdots&\vdots&\ddots&\vdots\\         a_{3m+2}&a_{3m+3}&\cdots&a_{4m+2}\\
     \end{bmatrix}
     \]
     have maximal rank. The $m$th Hessian is of the block-form
     \[
     \begin{bmatrix}
         0&0&0&a_1x_4\\
         0&\mathbf A&0&\mathbf x\\
         0&0&\mathbf B&\mathbf y\\
         a_1x_4&\mathbf x^t&\mathbf y^t&\ell_1\\
     \end{bmatrix}
     \]
     where
     \[
     \mathbf A = x_2 \begin{bmatrix}
         a_2&a_3&\cdots&a_{m+1}\\
         a_3&a_3&\cdots&a_{m}\\
\vdots&\vdots&\ddots&\vdots\\         a_{m+1}&a_{m+2}&\cdots&a_{2m}\\
     \end{bmatrix} + 
     x_4 \begin{bmatrix}
         a_3&a_4&\cdots&a_{m+2}\\
         a_4&a_5&\cdots&a_{m+2}\\
\vdots&\vdots&\ddots&\vdots\\         a_{m+2}&a_{m+2}&\cdots&a_{2m+1}\\
     \end{bmatrix}
     \]
     and \[
     \mathbf B = x_3 \begin{bmatrix}
         a_{2m+3}&a_{2m+4}&\cdots&a_{3m+2}\\
         a_{2m+4}&a_{2m+5}&\cdots&a_{3m+3}\\
\vdots&\vdots&\ddots&\vdots\\         a_{3m+2}&a_{3m+3}&\cdots&a_{4m+1}\\
     \end{bmatrix} + 
     x_4 \begin{bmatrix}
         a_{2m+4}&a_{2m+5}&\cdots&a_{3m+3}\\
         a_{2m+5}&a_{2m+6}&\cdots&a_{3m+4}\\
\vdots&\vdots&\ddots&\vdots\\         a_{3m+3}&a_{3m+4}&\cdots&a_{4m+2}\\
     \end{bmatrix}
     \]
     Hence the determinant of the $m$th Hessian is 
     \[
     (a_1x_4)^2\det\mathbf A\det \mathbf B
     \]
     which cannot vanish since the catalecticant matrix has maximal rank.
     \item[$(ix)$] When $[I]_2 = \langle x_1^2,x_1x_2,x_2^2,x_1x_4-x_2x_3\rangle$
     , the inverse system in degree $2m+1$ has a basis
given by    \[
    \begin{split}
    \left\{
    \frac{X_1X_3^{2m}}{(2m)!},\frac{X_1X_3^{2m-1}X_4}{(2m-1)!}+\frac{X_2X_3^{2m}}{(2m)!}, \dots,\frac{X_1X_4^{2m}}{(2m)!}+\frac{X_2X_3X_4^{2m-1}}{(2m-1)!}, \frac{X_2X_4^{2m}}{(2m)!},\right.\\ \left. \frac{X_3^{2m+1}}{(2m+1)!}, \frac{X_3^{2m}X_4}{(2m)!}, \dots, \frac{X_4^{2m+1}}{(2m+1)!}
    \right\}
    \end{split}
    \]
    and we have a monomial basis for $R/([I]_2)$ in degree $d$ given by 
    \[
    \{
    x_1x_3^d,x_1x_3^{d-1}x_4,\dots,x_1x_4^{d-1},x_2x_4^{d-1},x_3^{d},x_3^{d-1}x_4,\dots,x_4^d
    \}.
    \]
    Using these bases we get a catalecticant matrix 
    \[
    \begin{bmatrix}
        0&0&0&\cdots&0&a_1&a_2&a_3&\cdots&a_{m+2}\\
        0&0&0&\cdots&0&a_2&a_3&a_4&\cdots&a_{m+3}\\
        0&0&0&\cdots&0&a_3&a_4&a_5&\cdots&a_{m+4}\\
\vdots&\vdots&\vdots&\ddots&\vdots&\vdots&\vdots&\vdots&\ddots&\vdots\\
        0&0&0&\cdots&0&a_{m+1}&a_{m+2}&a_{m+3}&\cdots&a_{2m+2}\\
        a_1&a_2&a_3&\cdots&a_{m+2}&a_{2m+3}&a_{2m+4}&a_{2m+5}&\cdots&a_{3m+3}\\
        a_2&a_3&a_4&\cdots&a_{m+3}&a_{2m+4}&a_{2m+5}&a_{2m+6}&\cdots&a_{3m+4}\\
        a_3&a_4&a_5&\cdots&a_{m+4}&a_{2m+5}&a_{2m+6}&a_{2m+7}&\cdots&a_{3m+5}\\
\vdots&\vdots&\vdots&\ddots&\vdots&\vdots&\vdots&\vdots&\ddots&\vdots\\
        a_{m+1}&a_{m+2}&a_{m+3}&\cdots&a_{2m+2}&a_{3m+4}&a_{3m+5}&a_{3m+6}&\cdots&a_{4m+4}\\
    \end{bmatrix}
    \]
        The $m$th Hessian matrix is given by a block matrix of the form
    \[
    \begin{bmatrix}0&A\\A^t&B\end{bmatrix}
    \]
    where 
    \[ A = 
    x_3\begin{bmatrix}
        a_1&a_2&\cdots&a_{m+1}\\
        a_2&a_3&\cdots&a_{m+2}\\
         \vdots&\vdots&\ddots&\vdots\\
        a_{m+1}&a_{m+2}&\cdots&a_{2m+1}
    \end{bmatrix}
+    x_4\begin{bmatrix}
        a_2&a_3&\cdots&a_{m+2}\\
        a_3&a_4&\cdots&a_{m+3}\\
         \vdots&\vdots&\ddots&\vdots\\
        a_{m+2}&a_{m+3}&\cdots&a_{2m+2}
    \end{bmatrix}
    \]
    and the coefficients of $x_3^{m+1},x_3^{m}x_4,\dots,x_4^{m+1}$ in 
    $\det A$ are given by the maximal minors of 
    \[
    \begin{bmatrix}
    a_1&a_2&a_3&\cdots&a_{m+2}\\
    a_2&a_3&a_4&\cdots&a_{m+3}\\
    \vdots&\vdots&\vdots&\ddots&\vdots\\
    a_{m+1}&a_{m+2}&a_{m+3}&\cdots&a_{2m+2}\\
    \end{bmatrix}
    \]
    Hence the determinant of the Hessian matrix cannot vanish completely since the catalecticant matrix has full rank. This proves that WLP holds in this case.
    
    \noindent $(x)$  When $[I]_2 = \langle x_1^2,x_1x_2,x_2^2,x_1x_3\rangle$, we have that  that the inverse system of $([I]_2)$ in degree $d\ge 3$ has a monomial basis given by 
    \[
    \left\{
    \frac{X_1X_{4}^{d-1}}{(d-1)!},
    \frac{X_2X_3^{d-1}}{(d-1)!},
    \frac{X_2X_3^{d-2}X_4}{(d-2)!},\dots,
    \frac{X_2X_4^{d-1}}{(d-1)!},
    \frac{X_3^{d}}{d!}, \frac{X_3^{d-1}X_4}{(d-1)!}
    \frac{X_3^{d-2}X_4^2}{2!(d-2)!}, \dots, \frac{X_4^{d}}{d!}
    \right\}.
    \]
    We can use the corresponding monomial basis for the quotient $R/([I]_2)$ and we get the catalecticant matrix 
    \[
    \begin{bmatrix}
         0&0&0&\cdots&0&0&0&\cdots&0&{a}_{1}\\
      0&0&0&\cdots&0&{a}_{2}&{a}_{3}&\cdots&a_{m+2}&{a}_{m+3}\\
      0&0&0&\cdots&0&{a}_{3}&{a}_{4}&\cdots&{a}_{m+3}&{a}_{m+4}\\
\vdots&\vdots&\vdots&\ddots&\vdots&\vdots&\vdots&\ddots&\vdots&\vdots\\
0&0&0&\cdots&0&a_{m+1}&a_{m+2}&\cdots&a_{2m+1}&a_{2m+2}\\
      0&{a}_{2}&{a}_{3}&\cdots&{a}_{m+2}&{a}_{2m+3}&{a}_{2m+4}&\cdots&{a}_{3m+3}&{a}_{3m+4}\\
      0&{a}_{3}&{a}_{4}&\cdots&a_{m+3}&{a}_{2m+4}&{a}_{2m+5}&\cdots&{a}_{3m+1}&{a}_{3m+2}\\
\vdots&\vdots&\vdots&\ddots&\vdots&\vdots&\vdots&\ddots&\vdots&\vdots\\
0&a_{m+1}&a_{m+2}&\cdots&a_{2m+1}&a_{3m+2}&a_{3m+3}&\cdots&a_{4m+2}&a_{4m+3}\\      
      {a}_{1}&{a}_{m+2}&{a}_{m+3}&\cdots&{a}_{2m+2}&{a}_{3m+3}&{a}_{3m+4}&\cdots&{a}_{4m+3}&{a}_{4m+4}
    \end{bmatrix}  
    \]
    and we conclude that $a_1\ne 0$ and that
    \[
    \begin{bmatrix}
        a_2&a_3&\cdots&a_{m+2}\\
        a_3&a_4&\cdots&a_{m+3}\\
        \vdots&\vdots&\ddots&\vdots\\
        a_{m+1}&a_{m+2}&\cdots&a_{2m+1}\\
    \end{bmatrix}
    \]
    has maximal rank. 
    The $m$th Hessian matrix has the form 
    \[
    \begin{bmatrix}
        0 & \mathbf A \\ \mathbf A^t&\mathbf B
    \end{bmatrix}
    \]
    with 
    {\small
    \[
    \mathbf A = x_3\begin{bmatrix} 
    0&a_2&a_3&\cdots&a_{m+1}\\
    0&a_3&a_4&\cdots&a_{m+2}\\
    \vdots&\vdots&\vdots&\ddots&\vdots\\
    0&a_{m+1}&a_{m+2}&\cdots&a_{2m}\\
    0&a_{m+2}&a_{m+3}&\cdots&a_{2m+1}
    \end{bmatrix} + x_4\begin{bmatrix}
    0&a_3&a_4&\cdots&a_{m+2}\\
    0&a_4&a_5&\cdots&a_{m+3}\\
    \vdots&\vdots&\vdots&\ddots&\vdots\\
    0&a_{m+2}&a_{m+3}&\cdots&a_{2m+1}\\
    a_1&a_{m+3}&a_{m+4}&\cdots&a_{2m+2}
    \end{bmatrix}
    \] } 
    and 
    \[
    \det \mathbf A = (-1)^{m}a_1x_4 \det \mathbf A'
    \]
    where 
    \[
    \mathbf A' =x_3\begin{bmatrix}
    a_2&a_3&\cdots&a_{m+1}\\
    a_3&a_4&\cdots&a_{m+2}\\
    \vdots&\vdots&\ddots&\vdots\\
    a_{m+1}&a_{m+2}&\cdots&a_{2m}\\
    \end{bmatrix} + x_4\begin{bmatrix}
    a_3&a_4&\cdots&a_{m+2}\\
    a_4&a_5&\cdots&a_{m+3}\\
    \vdots&\vdots&\ddots&\vdots\\
    a_{m+2}&a_{m+3}&\cdots&a_{2m+1}\\
    \end{bmatrix}.
    \]
    The  coefficients of $x_3^m,x_3^{m-1}x_4,\dots,x_4^m$ in $\det \mathbf A'$ are the $m\times m$-minors of the matrix
    \[
     \begin{bmatrix}
    a_2&a_3&a_4&\cdots&a_{m+2}\\
    a_2&a_3&a_4&\cdots&a_{m+2}\\
    \vdots&\vdots&\vdots&\ddots&\vdots\\
    a_{m+1}&a_{m+2}&a_{m+3}&\cdots&a_{2m+1}\\
    \end{bmatrix}
    \]
    which we know has maximal rank. This shows that the $m$th Hessian doesn't vanish completely and we conclude that WLP holds also in this case.
\end{proof}

\begin{corollary}
    \label{main}
     Any artinian Gorenstein algebra  with $h$-vector $(1, 4, 6, 6, 4,  1)$ has the SLP .
\end{corollary}
\begin{proof}
By Corollary \ref{wlp slp}, it is enough to show that WLP holds and this follows from Theorem~\ref{keyresult}.
\end{proof}

We will end this section by proving some  partial results.

\begin{proposition}\label{aux}~
\begin{enumerate}
    \item 
    Any artinian Gorenstein algebra $A$ of socle degree $d\ge 2$ and  $h$-vector $(1,4,4,\dots,4,1)$ satisfies the SLP.    
    \item
    Any artinian Gorenstein algebra $A$ of  socle degree $d =2$ or $d\ge 4$ and $h$-vector $ (1, 5,5,\dots,5,1)$ satisfies the SLP.
    \item 
    Any artinian Gorenstein algebra $A$ of codimension $4$ and  socle degree $4$ satisfies the SLP.
\end{enumerate}        
\end{proposition}

\begin{proof}~ 
\begin{enumerate}
    \item 
 If $d=2$, $A$ is compressed with even socle degree and, therefore, it has the SLP. For $d\ge 3$, it is an immediate consequence of the Gordan-Noether theorem: for a general linear for $\ell \in [A]_1$ there is an isomorphism $\times\ell ^{d-2}:[A]_1\longrightarrow  [A]_{d-1}$. Hence $A$ satisfies the SLP.

\item 
If $d=2$, $A$ is compressed with even socle degree and therefore, it has the SLP.
For $d\ge 6$, it follows from Lemma~\ref{AADFIMMMN prop}, 
so it remains to prove the cases $d = 4$ and $d = 5$. Suppose that $A$ fails WLP, i.e., for any linear form  $\ell \in [A]_1$, the multiplication map 
\[
\times \ell : [A]_i \longrightarrow [A]_{i+1} 
\] 
is not bijective for  $\lceil \frac{d-1}{2} \rceil$.  Set 
$ B = A/(0:\ell)$ and $C = A/(\ell)$. 
By Green's theorem, we have that $\dim_K [C]_3\le 1$ which means that  WLP fails by at most $1$. This gives the following table of possibilities for the Hilbert functions of $B$ and $C$:
\[
\begin{array}{c|c|c}
d & B & C \\ \hline
4 & (1,4,4,1) & (1,4,1,1) \\
5 & (1,4,4,4,1) & (1,4,1,1,1) \\
\end{array}
\]
Let $\ell_1\in [A]_1$ be a general linear form and apply Lemma~\ref{snake}. We have that $\times\ell_1\colon[B]_i\longrightarrow [B]_{i+1}$ is an isomorphism by (1).  Since $C$ is a standard graded algebra and the last two non-zero graded components of $C$ are one-dimensional, also the multiplication map $\times\ell_1\colon[C]_{i+1}\longrightarrow [C]_{i+2}$
is an isomorphism. Hence we deduce, by Lemma~\ref{snake} , that $\times \ell _1:[A]_{i+1}\rightarrow [A]_{i+2}$ is bijective and by duality of Gorenstein algebras $\times \ell _1:[A]_i\rightarrow [A]_{i+1}$ is also bijective. Therefore, $A$ has the WLP. Since the Hilbert function is constant from degree $1$ to degree $d-1$, it follows that $\times \ell^{d-2}\colon [A]_1\longrightarrow[A]_{d-1}$ is bijective and $A$ satisfies the SLP.

\item  This follows immediately from the Gordan-Noether theorem (cf. Theorem~\ref{hesscriterium}).
\end{enumerate}
\end{proof}

\begin{proposition}\label{156651}
    Any artinian Gorenstein algebra $A$ of $h$-vector $( 1, 5, 6, 6, 5, 1)$ has the WLP. 
\end{proposition}

\begin{proof}
    Assume that $A$ fails WLP. Then for any linear form $\ell \in [A]_1$, the multiplication map $\times \ell : [A]_2\longrightarrow [A]_3 $ is not bijective. Fix a general linear form $\ell \in [A]_1$ and define  
\[
    B := A/(0:\ell)
    \hspace{2mm}\text{and}\hspace{2mm} C:= A/(\ell)
\]
recalling that $B$ is Gorenstein with socle degree one less than $A$. By Green's theorem we have that 
\[
0\le \dim_K[C]_3\le 1\quad\text{and}\quad 
0\le \dim_K[C]_4\le 1.
\]
Hence by the assumption that the WLP fails, $\dim_K[C]_3=1$.   
Using the short exact sequence 
\[
          0\to B(-1) \xrightarrow{\times \ell} A\to C\to 0
\]   
and the symmetry of the $h$-vector of $B$, there are only two possibilities for the $h$-vectors of $B$ and $C$.  
\[
\begin{array}{l|l|l}
 &B & C \\ \hline
 (i)&(1,5,5,5,1) & (1,4,1,1)\\
 (ii)&(1,4,5,4,1) & (1,4,2,1,1) \\
\end{array}
\]
Let $\ell_1\in [A]_1$ be a general linear form.
In case $(i)$, we have that $B$ satisfies the WLP by Proposition \ref{aux} (2) so the multiplication map $\times\ell_1\colon[B]_1\longrightarrow[B]_2$ is bijective. 
Since $\dim_K[C]_2=\dim_K[C]_3=1$ and the multiplication map $\times\ell_1\colon[C]_2\longrightarrow[C]_3$ is bijective. Lemma~\ref{snake} applies and the multiplication map $\times\ell_1\colon[A]_2\longrightarrow [A]_3$ is bijective, contradicting the  assumption that $A$ does not satisfy the WLP.

In case $(ii)$, we have that $B$ satisfies the WLP by Proposition \ref{aux} (3) so the multiplication map $\times\ell_1\colon[B]_2\longrightarrow[B]_3$ is surjective. 
Since $\dim_K[C]_3=\dim_K[C]_4=1$, the multiplication map $\times\ell_1\colon[C]_3\longrightarrow[C]_4$ is bijective. Lemma~\ref{snake} implies that the  map $\times\ell_1\colon[A]_3\longrightarrow [A]_4$ is surjective, contradicting the assumption that multiplication by a general linear form is not surjective from degree $3$ to degree $4$. Hence $A$ satisfies the WLP.
\end{proof}

\section{Main result}\label{S4}

In this section we will prove  the following theorem for 
artinian Gorenstein algebras.

\begin{theorem}\label{mainthm}
  Any artinian Gorenstein algebra $A$ of socle degree $d$ and Sperner number $\le d+1$ satisfies the WLP.  
\end{theorem}

\begin{remark}\label{optimal}
    The bound on the Sperner number given in the theorem above  is optimal since there are artinian Gorenstein algebras of any socle degree $d\ge 3$ with $h$-vector 
    \[
    (1,d+2,d+2,d+2,\dots,d+2,1)
    \]
    failing the WLP. We can obtain them as the trivial extension of the artinian level algebra $K[x,y]/(x,y)^{d}$ by its canonical module. As a quotient of the polynomial ring $R = K[x_1,x_2,\dots,x_{d+2}]$ this is given by the annihilator of the dual form 
    \[
    F = \sum_{i=1}^d X_iX_{d+1}^{d-1-i}X_{d+2}^i.
    \]
    
    In codimension $4$ it is optimal for $d=5$ and all $d\ge 7$ (see \cite[Prop. 3.5]{B}). It is still open whether WLP holds for all artinian Gorenstein algebras of socle degree $6$. (See \cite[\S 3.3]{AS})
    
    In codimension $5$, it is optimal for $d=3$ by the example above and for all  $d\ge 5$ by taking the trivial extension of a compressed level algebra of codimension $2$, socle degree $d-1$ and type $3$. However, it is not optimal for $d=4$ since artinian Gorenstein algebras with $h$-vector $(1,5,6,5,1)$ satisfy WLP by Corollary~\ref{low sperner} below.  
    
    In codimension $c\ge 6$ and higher, it is optimal for all $d\ge 3$ which we can see by taking the trivial extension of a compressed level algebra of codimension $2$, socle degree $d-1$ and type $c-2$.
\end{remark}

\begin{proof}[Proof of Theorem~\ref{mainthm}]
    We set $A = R/I$, let $\ell$ be a general linear form and let $B = A/(0:\ell)$ and $C = A/(\ell)$. 

    We start by proving the statement under the assumption that the socle degree is at most $5$. When the codimension is at most $3$, this follows from \cite[Corollary 3.12]{BMMNZ} and when the socle degree is at most $3$, it follows from Gordan-Noether. When the codimension is $4$, there are four possible $h$-vectors. 
    \begin{enumerate}
        \item  $(1,4,6,6,4,1)$. The WLP follows from Corollary~\ref{main}.
        \item $(1,4,5,5,4,1)$.  Green's theorem gives $\dim_K[C]_3=1$ if WLP fails, where $C = A/(\ell)$ for a general linear form. The $h$-vector of $B = A/(0\colon \ell)$ is $(1,4,4,4,1)$ and hence $B$ satisfies the WLP by Gordan-Noether. The $h$-vector of $C$ is $(1,3,1,1)$ and multiplication by a general linear form is bijective from degree $2$ to degree $3$. Hence  Lemma~\ref{snake} shows that $A$ satisfies the WLP.
        \item $(1,4,4,4,4,1)$. The WLP follows from Gordan-Noether.        
        \item $(1,4,5,4,1)$. The WLP follows from Gordan-Noether.
    \end{enumerate}
    When the codimension is $5$, there are three possible $h$-vectors. 
    \begin{enumerate}
        \item $(1,5,6,6,5,1)$. The WLP follows from Proposition~\ref{156651}.
        \item $(1,5,5,5,5,1)$. The WLP follows from Proposition~\ref{aux}.
        \item $(1,5,5,5,1)$. The WLP follows from Proposition~\ref{aux}.
    \end{enumerate}
    When the codimension is $6$, there is only one possible $h$-vector 
    \begin{enumerate}
        \item $(1,6,6,6,6,1)$. Green's theorem gives that $\dim_K[C]_3=1$ if the WLP fails, where $C = A/(\ell)$ for a general linear form $\ell$. The $h$-vector of $B = A/(0\colon \ell)$ is $(1,5,5,5,1)$ since it has to be unimodal. Hence $B$ satisfies the WLP by Proposition~\ref{aux}. The $h$-vector of $C$ is $(1,5,1,1,1)$ so multiplication by a general linear form is bijective from degree $2$ to degree $3$. By   Lemma~\ref{snake} we conclude that $A$ satisfies the $WLP$. 
    \end{enumerate}

It remains to prove the statement for socle degree $d\ge 6$.  Assume that the result is true for any artinian Gorenstein algebra  of socle degree $d\le d_0$, where $d_0\ge 5$, and we will prove it for any artinian Gorenstein algebra $A$ of socle degree $d = d_0+1$. 
Denote by $h_i=\dim[A]_i$, $h'_i=\dim [C]_k$ and $h''_i=\dim [B]_i$. Then $h_i=h''_{i-1}+h'_i$ for all $i>0$ and $h_0=h'_0=h''_0=1$. 

First assume that the map 
\[\times \ell\colon[A]_j\longrightarrow[A]_{j+1}\]
is not surjective for some $j\ge (d_0+1)/2$. By Lemma~\ref{lem:hilb is two} we have that $\dim_K[C]_{j+1}\le 1$, so the assumption implies that $\dim_K[C]_{j}\ge \dim_K[C]_{j+1}=1$ and hence the map $\times\ell_1\colon[C]_{j}\longrightarrow[C]_{j+1}$ is surjective for a general linear form $\ell_1$.

We observe that the Sperner number of $B$ is less than the Sperner number of $A$ since $h'_i>0$ for $i\le j+1$. Hence $B$ satisfies the hypothesis of the theorem and, by the induction hypothesis, $B$ satisfies the WLP and the map $\times\ell_1\colon[B]_{j}\longrightarrow[B]_{j+1}$ is surjective. By Lemma~\ref{snake} we conclude that the map $\times \ell_1:[A]_{j}\longrightarrow [A]_{j+1}$ is  surjective, contradicting the assumption. We conclude that the map $\times\ell\colon[A]_j\longrightarrow$ is surjective for $j\ge (d_0+1)/2$ and by duality, injective for $j\le (d_0-1)/2$.

Thus it  only remains to verify the surjectivity of the map 
\[
\times \ell\colon[A]_m\longrightarrow[A]_{m+1}
\]
for $m=n/2$ in the case when $n$ is even, which means that the socle degree of $A$ is $2m+1$.

From now on we will assume that $A$  has socle degree $2m+1$, Sperner number $\le 2m+2$ and it fails WLP only from degree $m$ to degree $m+1$. Hence multiplication by a general linear form $\ell$ is injective up to degree $m-1$ and surjective from degree $m+1$. We consider the difference $\delta:=h_{m}-h_{m-1}$, which is non-negative since $\times\ell\colon[A]_{m-1}\longrightarrow[A]_{m}$ is injective. We distinguish four cases:

\noindent$(i)$ If $\delta \ge 3$, we get that $h_m\ge 3m$ since $h'_i\ge \min\{i+1,3\}$, for $i=0,1,\dots,m$ and we have that multiplication by $\ell$ is injective up to degree $m$ in $A$. By the assumption on the Sperner number, we have that $h_m\le 2(m+1)$. Hence $3m\le 2(m+1)$ which contradicts the assumption that $m\ge 3$.

\noindent$(ii)$
When $\delta = 2$ we get that $h_i'\ge 2$ for $i=0,1,\dots,m$. Since multiplication by $\ell$ is injective on $A$ up to degree $m$, we get that $h'_i=h_i-h_{i-1}$, for $i=1,2,\dots,m$. Furthermore, $h'_1\ge 2$ and $h'_2\ge h'_3\ge \dots \ge h'_m =2$, while $1+h'_1+h'_2 +\cdots + h'_m  = h_m \le 2m+2$. This means that at most one of the $h'_i$ is $3$ and the remaining are $2$. Only $h'_1$  and $h'_2$ can be equal to $3$. 
Thus we have three possibilities for the $h$-vector: 
\begin{itemize}
    \item $(1,4,6,8,\dots,2m+2,2m+2,\dots, 10,8,6,4,1)$,
    \item $(1,3,6,8,\dots,2m+2,2m+2,\dots, 10,8,6,3,1)$
   or 
    \item $(1,3,5,7,\dots,2m+1,2m+1,\dots, 9,7,5,3,1)$. 
\end{itemize}
In these three cases, we apply Proposition~\ref{keyresult}, Propsition~\ref{136} and Proposition~\ref{135}.

\noindent$(iii)$ If $\delta = 1$ and $A$ only fails the WLP from degree $m$ to degree $m+1$ we obtain $h'_{m}=h'_{m+1}=1$ which implies that the Sperner number of $B$ is less than the Sperner number of $A$. The map $\times\ell_1\colon[C]_m\longrightarrow[C]_{m+1}$ is surjective by Green's theorem. $B$ satisfies the WLP by the induction hypothesis, so the map $\times\ell_1\colon[B]_{m-1}\longrightarrow[B]_{m}$ is surjective. Hence Lemma~\ref{snake} shows that the map $\times\ell_1\colon[A]_m\longrightarrow[A]_{m+1}$ is surjective.

\noindent $(iv)$ 
If $\delta =0$, we have $h_{m-1}=h_{m}=h_{m+1}=h_{m+2}$. Since the map $\times\ell\colon [A]_{m-1}\longrightarrow[A]_{m}$ is injective and $h_{m-1}=h_m$ we get $[C]_{m}=0$. Hence $[C]_j=0$ for $j\ge 0$ and also the map $\times\ell\colon[A]_{m}\longrightarrow[A]_{m+1}$ is  surjective. 
\end{proof}

Using our main result, we can now give a full classification of which $h$-vectors of Sperner number at most $6$  force the WLP for artinian Gorenstein algebras. 

\begin{corollary} \label{low sperner}  
Fix a Gorenstein $h$-vector $\underline{h}$ of Sperner number $\leq 6$. Then every Gorenstein algebra with Hilbert function $\underline{h}$ has the WLP if and only  if $\underline{h}$ is not one of the following $h$-vectors: $(1,5,5,1)$, $(1,6,6,1)$ or $(1,6,6,6,1)$.
\end{corollary}

\begin{proof}
From Theorem~\ref{mainthm} we have that the statement holds for if the socle degree is at least $5$. For socle degree at most $4$, WLP holds when the codimension is at most $3$ by \cite[Corollary 3.12]{BMMNZ}. In codimension $4$‚ Proposition~\ref{aux} (1) and (3) gives the result for socle degree at most $4$. When the socle degree is at most $2$, all artinian Gorenstein algebras have the WLP. It remains to consider the two $h$-vectors $(1,5,6,5,1)$ and $(1,5,5,5,1)$. In the second case, WLP follows from Proposition~\ref{aux} (2). Assume that the artinian Gorenstein algebra $A=R/I$ has Hilbert function $(1,5,6,5,1)$. Let $\ell$ be a general linear form, let $B = A/(0:\ell)$ and $C = A/(\ell)$. If $A$ does not satisfy the WLP we have that the $h$-vectors of $B$ and $C$ are $(1,4,4,1)$ and $(1,4,2,1)$ by Green's theorem. $B$ satisfies WLP by Gordan-Noether and for a general linear form $\ell_1$, the map $\times\ell_1\colon[C]_2\longrightarrow[C]_3$ is surjective by Green's theorem. Hence the map $\times\ell_1\colon[A]_2\longrightarrow[A]_3$ is surjective by Lemma~\ref{snake} and we conclude that $A$ satisfies the WLP.

The exceptions that do not force WLP can all be found, for instance, in \cite{G}. 
\end{proof}

\begin{remark}
    As mentioned in Remark~\ref{optimal}, it is open whether there are artinian Gorenstein algebras of codimension $4$ and socle degree $6$ not satisfying the WLP. See \cite[\S 3.3]{AS}.
\end{remark}


\begin{thebibliography}{99}

\bibitem{AS} 
 N.~Abdallah1, H. ~Schenck, {\em Free resolutions and Lefschetz properties of some Artin Gorenstein rings of codimension four}, Journal of Symbolic Computation 121 (2024), 102257.

\bibitem{AAIY} N. Abdallah, N. Altafi, A. Iarrobino and J. Yam\'eogo, {\em Jordan degree type for codimension three Gorenstein algebras of small Sperner number}, preprint arXiv:2406.06322v1.

\bibitem{AAISY}
N. Abdallah, N. Altafi, A. Iarrobino, A. Seceleanu, and
J. Yameogo, {\em 
Lefschetz properties of some codimension three artinian
Gorenstein algebras}, J. Algebra 625 (2023), 28--45.

\bibitem{wica} 
N. Altafi,  R. Dinu, S. Faridi, S. Masuti, R.M. Mir\'o-Roig, A. Seceleanu and  N. Villamizar, {\em WICA}. In preparation.

\bibitem{AADFIMMMN} N. Abdallah, N. Altafi, P. De Poi, L. Fiorindo, A. Iarrobino, P. Macias Marques, E. Mezzetti, R. M. Miró-Roig, L. Nicklasson,
{\em Hilbert functions and Jordan type of Perazzo artinian algebras}, Preprint arXiv:2303.16768  .

\bibitem{BGM} A. Bigatti, A.V. Geramita and J. Migliore, {\em Geometric consequences of extremal behavior in a theorem of Macaulay}, Trans. Amer. Math. Soc. {\bf 346} (1994), 203--235.

\bibitem{B} M. Boij, {\em Components of the space parametrizing graded {G}orenstein {A}rtin
  algebras with a given {H}ilbert function}. Pacific J. Math., {\bf 187(1)} (1999), 1--11.

  \bibitem{B2} M. Boij,  {\em Graded Gorenstein Artin algebras whose Hilbert functions have a large number of valleys}, Comm. Algebra 23(1995), no. 1, 97--103.

\bibitem{BMMNZ} M. Boij, J. Migliore, R. M. Mir\'o-Roig, U. Nagel and F. Zanello, On the Weak Lefschetz Property for artinian
Gorenstein algebras of codimension three, J. Algebra {\bf 403} (2014), 48–68.

\bibitem{BH} W. Bruns and J. Herzog, ``Cohen-Macaulay rings,'' Cambridge studies in advanced mathematics {\bf 39}, Revised edition (1998), Cambridge, U.K.


\bibitem{G}
R. Gondim, {\em On higher Hessians and the Lefschetz properties}, Journal of Algebra {\bf 489} (2017),
241–263.

\bibitem{GN}
P. Gordan and M. Noether, {\em Uber die algebraischen Formen deren Hesse’sche
Determinante identisch verschwindet}, Math. Ann. {\bf 10} (1876), 547–568.

\bibitem{Go} G. Gotzmann, \emph{Eine Bedingung f\"ur die Flachheit und das Hilbertpolynom cines graduierten Ringes}, Math. Z. \textbf{158} (1978), no. 1, 61--70.

\bibitem{Gr}
M.~Green, \emph{Restrictions of linear series to hyperplanes, and some results  of {M}acaulay and {G}otzmann}, Algebraic {C}urves and {P}rojective  {G}eometry, Lecture {N}otes in {M}athematics \textbf{1389}, Springer (1989), 76--86.

\bibitem{HMNW}
T. Harima, J. Migliore, U. Nagel and J. Watanabe, {\it The Weak and Strong Lefschetz properties for artinian $K$-algebras}, J. Algebra \textbf{262} (2003) 99--126.

\bibitem{IK}
A.~Iarrobino and V.~Kanev, \emph{Power sums, {G}orenstein {A}lgebras, and  {D}eterminantal {L}oci, With an {A}ppendix \lq\thinspace {T}he {G}otzmann {T}heorems and the {H}ilbert scheme}'  by {A}. {I}arrobino and {S}teven {L}.  {K}leiman, Lecture  Notes in Mathematics \textbf{1721}, Springer-Verlag, Berlin (1999).

\bibitem{K} S. Kleiman, {\em Bertini and his two fundamental theorems, Studies in the history
of modern mathematics, III.} Rend. Circ. Mat. Palermo (2) Suppl. {\bf 55} (1998), 9–37.

\bibitem{kustin} A. Kustin, {\em The Weak Lefschetz Property for standard graded Artinian Gorenstein algebras of embedding dimension four and socle degree three}, preprint arXiv:2404.08123v1.

\bibitem{MW} T. Maeno and J. Watanabe, {\em Lefschetz elements of artinian Gorenstein algebras and Hessians of
homogeneous polynomials}, Illinois J. Math. {\bf 53} (2009), 593--603.

\bibitem{MNZ} J. Migliore, U. Nagel and F. Zanello, {\em A characterization of Gorenstein Hilbert functions in codimension four with small initial degree}, Math. Res. Lett. {\bf 15} (2008), 331--349.

\bibitem{MZ} J. Migliore and F. Zanello, {\em The strength of the weak Lefschetz property}, Illinois Journal of Mathematics
{\bf 52} (2008),  Pages 1417--1433.

\bibitem{MZ2} J. Migliore and F. Zanello, {\em Stanley's nonunimodal Gorenstein $h$-vector is optimal}, Proc. Amer. Math. Soc. 145 (2017), 1--9.

\bibitem{MZ3} J. Migliore and F. Zanello, {\em The Hilbert functions which force the Weak Lefschetz Property}, J. Pure Appl. Algebra 210 (2007), 465--471.

\bibitem{RRR} L.~Reid, L.~G. Roberts, and M.~Roitman, {\em On
  complete intersections and their {H}ilbert functions},  Canad. Math. Bull., 34(4):525--535, 1991.

\bibitem{stanley} R. Stanley, {\it Hilbert functions of graded algebras}, Advances in Math. 28 (1978), no. 1, 57--83.
  
\bibitem{w1} J.\ Watanabe, {\em  A remark on the Hessian of homogeneous polynomials},
in: The Curves Seminar at Queen's, Volume XIII, Queen's Papers in Pure and Appl. Math. {\bf 119} (2000), 171--178.

\bibitem{WatdeB} J. Watanabe and M. de Bondt, {\em On the theory of Gordan-Noether
on homogeneous forms with zero Hessian (improved version)},
arXiv:math.AC/1703.07624, v.3.

 \end{thebibliography}
\end{document}